\documentclass[a4paper, 11 pt]{article}
\usepackage{a4wide, amsmath, amssymb, mathtools, yfonts}
\usepackage{comment}
\usepackage{tikz}
\usepackage{tikz-cd}
\usepackage[all]{xy}
\usepackage[utf8]{inputenc}
\usepackage{amsthm, mathrsfs}
\usepackage[english]{babel}
\usepackage{hyperref}
\usepackage{authblk}
\usepackage[OT2,T1]{fontenc}
\DeclareSymbolFont{cyrletters}{OT2}{wncyr}{m}{n}
\DeclareMathSymbol{\Sha}{\mathalpha}{cyrletters}{"58}

\numberwithin{equation}{section}
\newtheorem{lemma}{Lemma}[section]
\newtheorem{theorem}[lemma]{Theorem}

\newtheorem{corollary}[lemma]{Corollary}

\theoremstyle{definition}
\newtheorem{mydef}[lemma]{Definition}

\newcommand{\Z}{\mathbb{Z}}
\newcommand{\Q}{\mathbb{Q}}

\newcommand{\R}{\mathbb{R}}
\newcommand{\FF}{\mathbb{F}}

\newcommand{\Frob}{\textup{Frob}}

\newcommand\Gal{\mathrm{Gal}}

\newcommand{\res}{\textup{res}}
\newcommand{\inv}{\textup{inv}}

\title{\vspace{-\baselineskip}\sffamily\bfseries Elliptic curves of rank one over number fields}
\author[1]{Peter Koymans\thanks{Mathematisch Instituut, Universiteit Utrecht, Postbus 80.010, 3508 TA Utrecht, The Netherlands, p.h.koymans@uu.nl}}
\author[2, 3]{Carlo Pagano\thanks{Department of Mathematics and Statistics, Montreal, Quebec H3G 1M8, Canada and  Max Planck Institute for Mathematics, Bonn, Germany, Vivatsgasse 7, 53111, carlein90@gmail.com}}
\affil[1]{Utrecht University}
\affil[2]{Concordia University}
\affil[3]{Max Planck Institute for Mathematics}
\date{\today}

\begin{document}
\maketitle

\begin{abstract}
We prove that for every number field $K$, there exist infinitely many elliptic curves $E$ over $K$ with rank exactly equal to $1$.
\end{abstract}

\section{Introduction}
\subsection{Ranks of elliptic curves}
One of the simplest invariants of an elliptic curve $E$ over a number field $K$ is the rank of the Mordell--Weil group $E(K)$. Nevertheless, the set of possible values of this quantity is a source of great mystery. Even for $K := \Q$, it is unknown whether one can find curves with arbitrarily large rank or, alternatively, what the largest possible rank should be. In this latter direction, Elkies' longstanding record has been recently improved by Elkies--Klagsbrun. 

Likewise, as one starts varying the number field $K$, it is generally not known if a given number can occur as the rank of an elliptic curve. The current state of the art is that $0$ occurs as a rank over any number field by work of Mazur--Rubin \cite{MR}, but no result of this type is known for any fixed integer greater than $0$. It is a well-known folklore conjecture that every number field has an elliptic curve of rank exactly equal to $1$. We first present a corollary of our main result, which settles this conjecture. 

\begin{corollary}
\label{cRank1}
Let $K$ be a number field. Then there exist infinitely many elliptic curves $E$ over $K$ with rank $1$.
\end{corollary}

Corollary \ref{cRank1} was proven independently by Zywina \cite[Theorem 1.1]{Zywina2} extending his earlier techniques \cite{Zywina}. Zywina proceeds via isogeny descent on a non-isotrivial family. His main result \cite[Theorem 1.2]{Zywina} has the added benefit of providing a rank growth statement, while our main result Theorem \ref{tMain} has the benefit of applying to almost all elliptic curves with full rational $2$-torsion. Although both works proceed via additive combinatorics and descent, we emphasize that the inner mechanics of the respective works are rather distinct.

While our unconditional knowledge of large ranks is rather limited, it is widely believed that a typical elliptic curve should have low rank. This belief was made rigorous in the striking conjectures of Park--Poonen--Voight--Wood \cite{PPVW}. In particular, their work predicts that elliptic curves over a given number field should have uniformly bounded ranks. 

Their work also gives theoretical evidence for the so-called \emph{minimalist philosophy}, which postulates that in natural families of elliptic curves one should see most of the time the smallest possible value of the rank allowed by parity. A precise instance of the minimalist philosophy can be found in Silverman's conjecture \cite{SilConj} asserting that $100\%$ of the fibers in an elliptic fibration $\mathcal{E}$ on $\mathbb{P}^1$ should have rank equal to either the generic rank $\mathrm{rk} \, \mathcal{E}$ or $\mathrm{rk} \, \mathcal{E} + 1$. The alternative value $\mathrm{rk} \, \mathcal{E} + 1$ is there to address the possibility that the root number of a specialization has opposite sign than $(-1)^{\mathrm{rk} \, \mathcal{E}}$. Indeed, since the rank is at least the value of the generic rank for almost all fibers, BSD implies that the rank is then at least $\mathrm{rk} \, \mathcal{E} + 1$: hence the minimalist philosophy here asserts that the rank will, most of the time, settle for the smallest possible value that meets the adjusted parity. 

For the special case of the fibration over $\mathbb{P}^1 - \{0\}$, given by 
$$
ty^2 = f(x),
$$
for a cubic non-singular polynomial $f(x)$, one can explicitly describe the proportion of $t$ yielding positive and negative root number, and this depends on both $f$ and $K$. In this way one arrives at a generalized version of Goldfeld's conjecture, giving the precise probability for rank $0$ and rank $1$. If instead of the rank, one considers the cohomological avatar provided by the $2^{\infty}$-Selmer rank, then these conjectures have been established in groundbreaking work of Smith \cite{Smi22a, Smi22b} for general number fields under mild technical assumptions. Smith recently removed these technical assumptions if the base field is $\Q$, see \cite{Smi25}.

However, the $2^{\infty}$-Selmer rank being merely an upper bound to the rank leaves us with only partial results towards Goldfeld's conjecture. Namely, Goldfeld's prediction is correct within the positive root number portion of the curves: $100\%$ of them have rank equal to $0$. For negative root number, the situation is as follows: one knows that $100\%$ of them have rank at most $1$, and equal $1$ conditional on BSD (weaker, but still very deep, hypotheses such as parity would also suffice). However, without assuming BSD, it is presently unclear how to show that there is even a single elliptic curve of rank $1$ in the quadratic twist family above. Our main theorem proves that this is the case for a \emph{generic} elliptic curve with rational $2$-torsion. We introduce this notion now.

\begin{mydef}
We say that an elliptic curve $E/K$ is $n$-generic in case $E/K$ has a model of the form 
$$
E: y^2 = (x - a_1) (x - a_2) (x - a_3)
$$
with $a_1, a_2, a_3 \in O_K$ such that for all $\{i, j, k\} = \{1, 2, 3\}$ there exist at least $n$ principal primes $p$ such that:
\begin{itemize}
\item $p$ is totally positive, satisfies $p \equiv 1 \bmod 8O_K$ and satisfies $|O_K/p| > 5$,
\item $p$ divides $a_i - a_j$ with odd multiplicity, while $a_k - a_i$ (and therefore necessarily $a_k-a_j$) is an invertible square modulo $p$.
\end{itemize}
\end{mydef}

\noindent The terminology generic is justified as it is readily shown that, for each fixed $n \in \Z_{\geq 0}$, almost all $(a_1, a_2, a_3) \in O_K^3$ are $n$-generic when ordered by Weil height.

\begin{theorem}
\label{tMain}
Let $K$ be a number field and let $E$ be an elliptic curve over $K$ such that $E(K)[2] \cong \mathbb{F}_2^2$. Assume that $E$ is $3$-generic. 

Then there exist infinitely many $t \in K^\ast/K^{\ast 2}$ such that $\mathrm{rk} \, E^t(K) = 1$.
\end{theorem}

We regard it as a very interesting challenge to optimize Theorem \ref{tMain} and prove it for all elliptic curves with full rational $2$-torsion for which the complete distribution of $2^{\infty}$-Selmer is known (for the state of the art, see \cite{Smi22b}). 

\subsection{Method of proof}
It is not hard to construct elliptic curves with positive rank. In fact, one can explicitly write down plenty of families where the generic rank is positive. It is also not too difficult, for a given elliptic curve, to provide upper bounds on the rank using descent. Until recently, however, it was unknown how to systematically combine the two approaches and \emph{simultaneously} perform descent on a family with generic rank arranged to be positive. The main difficulty comes down to the fact that descent involves the prime factorization of the discriminant: it is typically difficult to control prime factors of polynomials.

A recent insight of \cite{KP} is to bring additive combinatorics into the field, specifically a result of Kai \cite{Kai} (which is based on the earlier works \cite{GT0, GT1, GT2, GT3} over the integers). An immediate application of \cite{KP} is a resolution of Hilbert's tenth problem for finitely generated infinite $\Z$-algebras. The key technical result in \cite{KP} is that whenever $K$ has at least $32$ real places, there exists an elliptic curve $E/K$ such that 
$$
\mathrm{rk} \, E(K(i)) = \mathrm{rk} \, E(K) > 0.
$$
By well-known techniques (see \cite{Poonen, Shl}), such a result resolves Hilbert's tenth problem by reducing to the case of $\Z$ which was famously settled by Matiyasevich \cite{Ma} in 1970. 

In its broad strokes, the strategy of combining descent with additive combinatorics has also been applied in various other settings. Alp\"{o}ge--Bhargava--Ho--Shnidman \cite{ABHS} gave an elegant alternative proof of the undecidability of Hilbert's tenth problem using a rather different abelian variety and set of twists than \cite{KP}. Zywina \cite{Zywina} also used the same high-level strategy to show that there are infinitely many elliptic curves of rank $2$ over $\Q$. A nice feature of \cite{Zywina} is that the family of elliptic curves considered is not isotrivial, which requires a different input from both the additive combinatorics part and the descent part (using $2$-isogeny descent in place of $2$-descent).

One limiting feature of the work \cite{KP} is that the method leveraged in an important way on the presence of the $32$ real embeddings of $K$, so its scope does not cover Corollary \ref{cRank1} or Theorem \ref{tMain}. It is the purpose of the present paper to provide a version of the method in \cite{KP} that works for all number fields without relying on the presence of auxiliary real places. The main insight of the present work, which leads to this broader agility and applicability, is that the places of split multiplicative reduction for the original curve play the role of a \emph{non-archimedean analogue} of the infinite places used in \cite{KP}. 

The difficulty arises from combining \emph{pre-twisting} with additive combinatorics. The purpose of pre-twisting comes from the fact that the size of Selmer groups changes by a bounded amount for each prime involved. Hence it is not always possible to bring the Selmer rank down to a prescribed level merely using the primes provided by additive combinatorics: it is pre-twisting which arranges the Selmer structure favorably before the prime values of polynomials enter the scene.

However, the additive combinatorics primes turn out to have all the same congruence modulo the pre-twisting primes. Meanwhile, the final Selmer structure is dictated by the behavior of the pre-twisting primes modulo the additive combinatorics primes. For this reason, we would like to have enough degrees of freedom in swapping quadratic symbols. In \cite{KP}, this freedom comes from delicate manipulations with real places, while this is provided by analogous manipulations with places of split multiplicative reduction in the present work. 

\subsection{Applications}
The existence of curves as in Corollary \ref{cRank1} plays an important role in the field of definability and decidability over arithmetically interesting rings. A number of authors have proved results conditionally on the truth of Corollary \ref{cRank1}, particularly in the context of attacking Hilbert's tenth problem over so-called \emph{big rings}. These are subrings of $\overline{\mathbb{Q}}$ that are beyond the finitely generated case, and have infinitely many primes appearing in denominators of its elements. 

In \cite[Theorem 1.8]{Eis-Ev-Shl}, Eisentr\"{a}ger--Everest--Shlapentokh proved, among other things, that if Corollary \ref{cRank1} holds, then for any decomposition 
$$
\delta_1 + \ldots + \delta_t = 1
$$
where $t>1$ is an integer, and $\delta_1, \ldots, \delta_t$ are computable positive real numbers, one can find a partition $S_1, \ldots, S_t$ of the finite places of $K$, where each $S_i$ has natural density equal to $\delta_i$, and such that $\Z$ admits a diophantine model over $O_K[S_i^{-1}]$. In particular, Hilbert's tenth problem has a negative answer for each of the rings $O_K[S_i^{-1}]$. They also proved (see \cite[Theorem 1.7]{Eis-Ev-Shl}) that, conditionally on the truth of Corollary \ref{cRank1}, one can find a partition $P_1, \ldots, P_t$ of the finite places of $K$, still with each $P_i$ of density $\delta_i$, such that in each $O_K[P_i^{-1}]$ one has a diophantine subset that is discrete under all absolute values of $K$. Their work built on earlier work of Poonen--Shlapentokh \cite{Poonen--Shl}, which was in turn inspired by previous work of Poonen \cite{Poonen: definability}. In a similar spirit, earlier work of Cornelissen--Shlapentokh \cite{Cornelissen--Shl} established that Corollary \ref{cRank1} implies that for any number field $K$ one can find a sequence of sets of primes $S_n$ with density converging to $1$, as $n$ goes to infinity, with the property that $\Z$ can be defined over $O_K[S_n^{-1}]$ using at most two universal quantifiers. 


\subsection{Layout of the paper}
In Section \ref{sct: background}, we provide some background facts on elliptic curves. Section \ref{sct: add comb} is the heart of the paper and gives the proof of Theorem \ref{tMain}. The proof of Theorem \ref{thm: auxiliary implies suitable} is the most critical step in Section \ref{sct: add comb}. This reduces Theorem \ref{tMain} to showing that there exists an auxiliary twist, which is then proven in Theorem \ref{tFinal}. In the process of proving Theorem \ref{tFinal}, we give a simplified treatment with respect to \cite{KP}, in particular we bypass a laborious parametrization of $H^1(G_K,\mathbb{F}_2)$. Finally, in Section \ref{sct: proof of cor} we deduce Corollary \ref{cRank1} from Theorem \ref{tMain}. 

\subsection*{Acknowledgements}
The authors would like to thank Stephanie Chan, Gunther Cornelissen, Andrew Granville, Sun Woo Park, Ren\'e Schoof and Alexandra Shlapentokh for valuable discussions. We would also like to thank Adam Morgan for explaining to us how to control the 2-Selmer group with the Markov chain from Subsection \ref{sMarkov} during our shared time at the Max Planck Institute for Mathematics in 2019-2020. The first author gratefully acknowledges the support of the Dutch Research Council (NWO) through the Veni grant ``New methods in arithmetic statistics''. The second author wishes to thank the Max Planck Institute for Mathematics in Bonn for its financial support, great work conditions and an inspiring atmosphere.

\section{Background theory} \label{sct: background}
\subsection{Torsion points}
In our next section, we will use additive combinatorics to construct a rational point on certain quadratic twists of a fixed elliptic curve $E$ satisfying the assumptions of Theorem \ref{tMain}. In order to ensure that this point is not torsion, we will use the following lemma.

\begin{lemma}
\label{lTorsion}
Let $K$ be a number field and let $E/K$ be an elliptic curve. Then we have for all but finitely many $d \in K^\ast/K^{\ast 2}$
$$
E^d(K)^{\textup{tors}} = E^d(K)[2].
$$
\end{lemma}

\begin{proof}
See \cite[Lemma 3.2]{KP}.
\end{proof}

\subsection{Root numbers}
There are various ways in the literature to control the parity of $2$-Selmer ranks. This can be done directly, but since the $2$-parity conjecture is known in our situation, we have opted to do this via root numbers. The global root number $w(E/K)$ is by definition
$$
w(E/K) = \prod_v w(E/K_v),
$$
where $w(E/K_v) \in \{-1, +1\}$ is the local root number. We shall fortunately not need the precise definition of $w(E/K_v)$: the following facts will be sufficient for our work.

\begin{lemma}
\label{lRootFacts}
Let $\mathcal{K}$ be a non-archimedean local field of characteristic $0$. Write $k$ for the residue field of $\mathcal{K}$ and $v$ for the associated valuation. Then:
\begin{itemize}
\item[(i)] If $E/\mathcal{K}$ has good reduction, then we have $w(E/\mathcal{K}) = +1$.
\item[(ii)] If $E/\mathcal{K}$ has additive, potentially good reduction and $\mathrm{char}(k) \geq 5$, then we have
$$
w(E/\mathcal{K}) = (-1)^{\lfloor \frac{v(\Delta) |k|}{12} \rfloor}.
$$
\item[(iii)] If $E/\mathcal{K}$ has split multiplicative reduction, then $w(E/\mathcal{K}) = -1$.
\item[(iv)] If $E/\mathcal{K}$ has non-split multiplicative reduction, then $w(E/\mathcal{K}) = +1$.
\end{itemize}
\end{lemma}

\begin{proof}
This is a special case of \cite[Theorem 2.3]{CD}.
\end{proof}

Our next result is the relevant (known) case of the $2$-parity conjecture that we will use.

\begin{lemma}
\label{lRootNumber}
Let $K$ be a number field and let $E/K$ be an elliptic curve. Assume that $E(K)[2] \cong \mathbb{F}_2^2$. Then we have
$$
(-1)^{\dim_{\mathbb{F}_2} \mathrm{Sel}^2(E/K)} = w(E/K).
$$
\end{lemma}

\begin{proof}
See \cite[Corollary 3.6]{KP}.
\end{proof}

\subsection{Selmer ranks as a Markov chain}
\label{sMarkov}
We will now encode the $2$-Selmer rank of $E^d$ as a Markov chain. As the necessary theory was already set up in \cite[Section 4]{KP} (which is in turn based on the earlier works \cite{PR} and \cite{KMR, KMR2}), we will be rather brief about this. 

Throughout this subsection, $K$ denotes a fixed number field and $E$ denotes a fixed elliptic curve with $E(K)[2] \cong \mathbb{F}_2^2$. We write $\Omega_K$ for the set of all places of $K$, and we fix a finite subset $T \subseteq \Omega_K$ containing all infinite places, all $2$-adic places and all places of bad reduction of $E$.

\begin{mydef}
Let $M$ be a finite $G_K$-module. A Selmer structure $\mathcal{L} = (\mathcal{L}_v)_{v \in \Omega_K}$ is a sequence of subgroups $\mathcal{L}_v \subseteq H^1(G_{K_v}, M)$ such that $\mathcal{L}_v = H^1_{\textup{nr}}(G_{K_v}, M)$ for all but finitely many $v$. Given a Selmer structure $\mathcal{L}$, the corresponding Selmer group is by definition
$$
\mathrm{Sel}_{\mathcal{L}}(G_K, M) := \ker\left(H^1(G_K, M) \rightarrow \prod_{v \in \Omega_K} \frac{H^1(G_{K_v}, M)}{\mathcal{L}_v}\right).
$$
\end{mydef}

We are now ready to define a sequence of Selmer structures, which we roughly think of as a Markov chain on the set of places of bad reduction of $E^d$. This sequence has two pleasant features: it is easy to describe the change of dimension in each step, and there is a simple relation between the $2$-Selmer rank of the elliptic curve and the (second to) last Selmer group of the sequence.

We let
$$
E: y^2 = (x - a_1) (x - a_2) (x - a_3)
$$
be a model of $E$ with $a_1, a_2, a_3 \in K$ distinct. We define $\alpha := a_1 - a_2$, $\beta := a_1 - a_3$ and $\gamma := a_2 - a_3$. Write $(-,-)_{\text{Weil}}$ for the Weil pairing and write $P_1 = (a_1, 0)$ and $P_2 = (a_2, 0)$. Then we get two linear functionals $\lambda_1, \lambda_2: E[2] \rightarrow \mathbb{F}_2$ given by
$$
\lambda_i: P \mapsto (P, P_i)_{\text{Weil}}.
$$
The pair of functionals $(\lambda_1, \lambda_2): E[2] \rightarrow \mathbb{F}_2^2$ is an isomorphism. Henceforth we shall frequently identify $E[2]$ with $\mathbb{F}_2^2$, and the implicit identification will always be $(\lambda_1, \lambda_2)$.

Let $i \in \Z_{\geq 0}$ be an integer and let $v_1, \dots, v_i$ be distinct places outside of $T$. For each integer $1 \leq j \leq i$, we let $\pi_j$ denote a choice of an element of $K_{v_j}^\ast/K_{v_j}^{\ast 2}$ with odd valuation: observe that there are exactly two such choices. We write $\boldsymbol{\pi} = (\pi_j)_{1 \leq j \leq i}$ for the corresponding vector, and we define the Selmer structure $\mathcal{L}_{i, \boldsymbol{\pi}} = (\mathcal{L}_{i, \boldsymbol{\pi}, v})_{v \in \Omega_K}$ through
$$
\mathcal{L}_{i, \boldsymbol{\pi}, v} = 
\begin{cases}
\delta(E(K_v)) &\text{if } v \in T, \\
\langle (\alpha \beta, \pi_j \alpha), (-\pi_j \alpha, - \alpha \gamma) \rangle &\text{if } v = v_j \text{ for some } j \in \{1, \dots, i\}, \\
H^1_{\text{nr}}(G_{K_v}, E[2]) &\text{otherwise.}
\end{cases}
$$
Here we have written $\delta$ for the connecting map induced by the Kummer sequence
$$
0 \rightarrow E[2] \rightarrow E \xrightarrow{\cdot 2} E \rightarrow 0.
$$
The Selmer group $\mathrm{Sel}_{\mathcal{L}_{i, \boldsymbol{\pi}}}(G_K, E[2])$ is neatly connected to the $2$-Selmer group of $E^d$, as our next result codifies. We write $\psi_d \in H^1(G_K, \mathbb{F}_2)$ for the quadratic character attached to $d \in K^\ast/K^{\ast 2}$ via Kummer theory.

\begin{lemma}
\label{lFinalStep}
Let $K$, $E$ and $T$ be as above. Let $d \in K^\ast/K^{\ast 2}$ and assume that
\begin{equation}
\label{eLocalT}
\mathrm{res}_v(\psi_d) = 0 \quad \quad \textup{ for all } v \in T.
\end{equation}
Let $r \geq 1$ be an integer and let $v_1, \dots, v_r$ be the places ramified in $\psi_d$. Let $\pi_1, \dots, \pi_{r - 1}$ be such that $\mathrm{res}_{v_i}(\psi_d) = \psi_{\pi_i}$ for all $i$. Then we have
$$
\dim_{\FF_2} \mathrm{Sel}^2(E^d/K) = 2 + \dim_{\FF_2} \mathrm{Sel}_{\mathcal{L}_{r - 1, \boldsymbol{\pi}}}(G_K, E[2]).
$$
\end{lemma}

\begin{proof}
This is \cite[Lemma 4.2]{KP}.
\end{proof}

From now on, we shall frequently abuse notation by using $\mathrm{Sel}_{\mathcal{L}_{i, \boldsymbol{\pi}}}(G_K, E[2])$ also when $\boldsymbol{\pi}$ is a vector of length greater than $i$ (with the implicit meaning that one should use only the first $i$ entries of $\boldsymbol{\pi}$). In the remainder of this subsection, we will explain, given another place $v_{i + 1}$ outside $T \cup \{v_1, \dots, v_i\}$ and an element $\pi_{i + 1} \in K_{v_{i + 1}}^\ast/K_{v_{i + 1}}^{\ast 2}$, how to compute the change of Selmer rank
$$
n_i := \dim_{\FF_2} \mathrm{Sel}_{\mathcal{L}_{i + 1, \boldsymbol{\pi}}}(G_K, E[2]) - \dim_{\FF_2} \mathrm{Sel}_{\mathcal{L}_{i, \boldsymbol{\pi}}}(G_K, E[2]).
$$
To this end, we introduce a second Selmer structure $\mathcal{L}_{i, \boldsymbol{\pi}}' = (\mathcal{L}_{i, \boldsymbol{\pi}, v}')_{v \in \Omega_K}$ given by
$$
\mathcal{L}_{i, \boldsymbol{\pi}, v}' = 
\begin{cases}
\delta(E(K_v)) &\text{if } v \in T, \\
\langle (\alpha \beta, \pi_j \alpha), (-\pi_j \alpha, - \alpha \gamma) \rangle &\text{if } v = v_j \text{ for some } j \in \{1, \dots, i\}, \\
H^1(G_{K_v}, E[2]) &\text{if } v = v_{i + 1} \\
H^1_{\text{nr}}(G_{K_v}, E[2]) &\text{otherwise}
\end{cases}
$$
and
$$
A_{v_{i + 1}} = \mathrm{res}_{v_{i + 1}} \mathrm{Sel}_{\mathcal{L}_{i, \boldsymbol{\pi}}'}(G_K, E[2]).
$$

\begin{lemma}
\label{lSelmerChange}
Let $K$, $E$ and $T$ be as above. Let $i \in \Z_{\geq 0}$ and let $v_1, \dots, v_{i + 1}$ be distinct places outside of $T$. Then we have
$$
n_i = 
\begin{cases}
2 &\textup{if } \dim_{\FF_2} \res_{v_{i + 1}}(\mathrm{Sel}_{\mathcal{L}_{i, \boldsymbol{\pi}}}(G_K, E[2])) = 0 \textup{ and } A_{v_{i + 1}} = \mathcal{L}_{i + 1, \boldsymbol{\pi}, v_{i + 1}} \\
-2 &\textup{if } \dim_{\FF_2} \res_{v_{i + 1}}(\mathrm{Sel}_{\mathcal{L}_{i, \boldsymbol{\pi}}}(G_K, E[2])) = 2, \\
0 &\textup{otherwise.}
\end{cases}
$$
\end{lemma}

\begin{proof}
See \cite[Lemma 4.1]{KP}, which is based on earlier work of \cite{DD} and further refined by \cite{Ces}.
\end{proof}

Finally, we shall need the following precise description of the local $2$-Selmer spaces for primes of split multiplicative reduction. 

\begin{lemma} 
\label{lSelmeratBad}
Let $K$ be a non-archimedean local field of characteristic $0$ with a valuation $v$ satisfying $|O_K/v| \equiv 1 \bmod 4$. Let $\pi_v$ be a uniformizer and let $\epsilon_v$ be any non-square element of $O_K^\ast$. Suppose that $a_1, a_2, a_3 \in O_K$. Then the following hold:
\begin{enumerate}
\item[$(i)$] If $v(\alpha) \equiv 1 \bmod 2$, while the elements $\beta, \gamma$ are invertible squares modulo $v$, then
$$
\delta(E(K)) = \langle (\pi_v, \pi_v), (\epsilon_v, \epsilon_v) \rangle.
$$
\item[$(ii)$] If $v(\beta) \equiv 1 \bmod 2$, while the elements $\alpha, \gamma$ are invertible squares modulo $v$, then
$$
\delta(E(K)) = \langle (\pi_v, 1), (\epsilon_v, 1) \rangle.
$$
\item[$(iii)$] If $v(\gamma) \equiv 1 \bmod 2$, while the elements $\alpha, \beta$ are invertible squares modulo $v$, then 
$$
\delta(E(K)) = \langle (1, \pi_v),(1, \epsilon_v) \rangle.
$$
\end{enumerate}
\end{lemma}

\begin{proof}
In general, $\delta(E(K))$ is maximal isotropic for the quadratic form $q_E$ defined by
$$
(x, y) \mapsto (\alpha \beta \cdot x, - \alpha \gamma \cdot y)_v.
$$ 
We now prove $(i)$, the other cases being similar. Denote by $\pi$ the (ramified) class of $\alpha$ in $K^\ast/K^{\ast 2}$. Observe that $(\pi, \pi) \in \delta(E(K))$ since we have $(\pi, \pi) \in \delta(E(K)[2])$. In case $(i)$, $\delta(E(K))$ is then maximal isotropic for the quadratic form $q_E$ defined by
$$
(x, y) \mapsto (\pi \cdot x, \pi \cdot y)_v.
$$ 
The associated bilinear form sends $((x, y), (x', y'))$ to $(x, y')_v (x', y)_v$. With this, we see that the only non-trivial unramified class orthogonal to $(\pi, \pi)$ is the class $(\epsilon_v, \epsilon_v)$. This class vanishes under $q_E$ as well, thus yielding a maximal isotropic space. We conclude that if $\delta(E(K))$ contains a non-trivial unramified class, then it has to be as in the conclusion of the lemma.

Hence it suffices to show that there is a non-trivial unramified class. The existence of such a class is equivalent to the existence of a point $P \in E(K) - 2E(K)$ such that $P \in 2E(L)$, where $L/K$ is the unique unramified quadratic extension of $K$. Therefore it remains to establish the existence of such a point $P$, which we do next. 

We start by observing that $E$ has split multiplicative reduction. Indeed, \cite[p.~186]{Silverman} implies that the equation is minimal, and furthermore we know that $\beta$ and $\gamma$ are squares modulo $v$. Therefore, from Tate's parametrization, there exists $q \in K^\ast$ with positive valuation and a Galois equivariant isomorphism
$$
\varphi: \overline{K}^\ast/q^{\Z} \rightarrow E(\overline{K}).
$$
Endowing $\Z$ with the trivial action of $G_K$, we have $H^1(G_K, \Z) = 0$ because there are no non-trivial continuous homomorphisms from a profinite group to $\Z$. Hence we conclude that
$$
E(K) \cong \left(\overline{K}^\ast/q^{\Z}\right)^{G_K} = K^\ast/q^{\Z}, \quad \quad E(L) \cong \left(\overline{K}^\ast/q^{\Z}\right)^{G_L} = L^\ast/q^{\Z}.
$$
The desired point $P$ is now concretely given by $\varphi(\epsilon_v)$. Indeed, the class of $\epsilon_v$ inside $K^\ast/q^{\Z}$ becomes divisible by $2$ in $L^\ast/q^{\Z}$ but it is definitely not divisible by $2$ in $K^\ast/q^{\Z}$.
\end{proof}

\subsection{Finite fields}
In order to construct our linear forms (to which we will apply Kai's theorem \cite[Theorem A.8]{Kai}), we will use a technical result on finite fields.

\begin{lemma} 
\label{lfinitefields}
Let $\mathbb{F}_q$ be a finite field with $q > 5$ odd, and let $c_1, c_2, c_3, c_4, c_5, c_6 \in \mathbb{F}_q$. Assume that $c_1 c_4 - c_2 c_3 \neq 0$, $c_1 c_6 - c_2 c_5 \neq 0$ and $c_3 c_6 - c_4 c_5 \neq 0$. Then for all choices of $\delta_1, \delta_2, \delta_3 \in \mathbb{F}_q^\ast$ and all $\lambda_1, \lambda_2, \lambda_3 \in \mathbb{F}_q$, there exist $u, v \in \mathbb{F}_q$ and $s_1, s_2, s_3 \in \mathbb{F}_q^\ast$ such that
\begin{align*}
c_1 u + c_2 v &= \delta_1 s_1^2 + \lambda_1 \\
c_3 u + c_4 v &= \delta_2 s_2^2 + \lambda_2 \\
c_5 u + c_6 v &= \delta_3 s_3^2 + \lambda_3.
\end{align*}
\end{lemma}

\begin{proof}
For $s_1$ and $s_2$ to be chosen later, we set
\begin{align*}
u_0(s_1, s_2) &:= \frac{c_4 (\delta_1 s_1^2 + \lambda_1) - c_2 (\delta_2 s_2^2 + \lambda_2)}{c_1 c_4 - c_2 c_3} \\
v_0(s_1, s_2) &:= \frac{-c_3 (\delta_1 s_1^2 + \lambda_1) + c_1 (\delta_2 s_2^2 + \lambda_2)}{c_1 c_4 - c_2 c_3}.
\end{align*}
Then $u_0(s_1, s_2)$ and $v_0(s_1, s_2)$ satisfy the first two equations by construction. Hence it suffices to show that for all $\delta_3 \in \mathbb{F}_q^\ast$ and all $\lambda_3 \in \mathbb{F}_q$, there exist $s_1, s_2, s_3 \in \mathbb{F}_q^\ast$ such that
$$
c_5 u_0(s_1, s_2) + c_6 v_0(s_1, s_2) = \delta_3 s_3^2 + \lambda_3.
$$
The above equation is equivalent to
$$
(c_4 c_5 - c_3 c_6) (\delta_1 s_1^2 + \lambda_1) + (c_1 c_6 - c_2 c_5) (\delta_2 s_2^2 + \lambda_2)  = (c_1 c_4 - c_2 c_3) (\delta_3 s_3^2 + \lambda_3).
$$
Renaming variables, we need to show that for every fixed $\epsilon_1, \epsilon_2, \epsilon_3 \in \mathbb{F}_q^\ast$ and fixed $\lambda \in \mathbb{F}_q$, there exist $s_1, s_2, s_3 \in \mathbb{F}_q^\ast$ with
$$
\epsilon_1 s_1^2 + \epsilon_2 s_2^2 = \epsilon_3 s_3^2 + \lambda.
$$
But this follows easily from the Cauchy--Davenport theorem and our assumption $q > 5$.
\end{proof}

\section{Additive combinatorics} 
\label{sct: add comb}
The goal of this section is to prove Theorem \ref{tMain}. We will do this using two reduction steps. These are respectively Theorem \ref{tInfSuitable}, which gives sufficient conditions for rank $1$ (called a suitable twist), and Theorem \ref{thm: auxiliary implies suitable}, which gives sufficient conditions for infinitely many suitable twists provided that we have a certain auxiliary pre-twist. The existence of this pre-twist is established in Theorem \ref{tFinal}.

\subsection{Reduction steps}
Take $E$ to be a $3$-generic elliptic curve as in Theorem \ref{tMain}, so $E$ has the shape
$$
E: y^2 = (x - a_1) (x - a_2) (x - a_3)
$$
with $a_1, a_2, a_3 \in O_K$. 

\begin{lemma}
\label{lGenericwithoddrank}
Let $n \in \Z_{\geq 1}$ and let $E$ be a $n$-generic elliptic curve over $K$. Then there exists $t_0 \in K^{\ast}/K^{\ast 2}$ such that $E^{t_0}$ is $(n - 1)$-generic and has root number $-1$. 
\end{lemma}

\begin{proof}
Clearly, we may assume that $w(E/K) = +1$. Since $E$ is $n$-generic, we may fix a prime $v$ with $v(\alpha) \equiv 1 \bmod 2$ and $v(\beta) = v(\gamma) = 0$. Observe that $v$ must be of split multiplicative reduction for $E$ by \cite[Proposition 5.1, Chapter 7]{Silverman} combined with the fact that $\beta, \gamma$ are invertible squares modulo $v$. 

By the Mitsui prime ideal theorem, there exists a prime element $q \nmid 6$ that is totally positive, an invertible square at all places dividing 6, a square at all places of bad reduction of $E$ except for $v$ and a non-square in $(O_K/v)^\ast$. We claim that $w(E^q/K) = -1$. Since $E^q$ is $(n - 1)$-generic, the claim clearly implies the lemma. So it remains to show that $w(E^q/K) = -1$.

Note that the places of bad reduction for $E^q$ are the disjoint union of $(q)$ and the places of bad reduction of $E$. Moreover, at all places $v'$ of bad reduction of $E$ except for $v$, we have $E \cong E^q$ locally at $v'$. Hence Lemma \ref{lRootFacts}, specifically part $(i)$, $(iii)$ and $(iv)$, shows that
\begin{align}
\label{eRootChange}
w(E^q/K) w(E/K) = w(E^q/K_{(q)}) w(E^q/K_v) w(E/K_v) = -w(E^q/K_{(q)}),
\end{align}
since $E$ has split multiplicative reduction at $v$ and $E^q$ has non-split multiplicative reduction at $v$. 

We now make the subclaim that $w(E^q/K_{(q)}) = +1$, which implies the desired conclusion thanks to equation \eqref{eRootChange} and our assumption $w(E/K) = +1$. So it remains to prove the subclaim. Observe that $E^q$ has additive, potentially good reduction at $q$. Hence Lemma \ref{lRootFacts}$(iv)$ shows that our subclaim is equivalent to $|(O_K/q)^\ast| \equiv 1 \bmod 4$. But this follows from Hilbert reciprocity applied to the $2$-cocycle $\psi_q \cup \psi_{-1}$.
\end{proof}

By Lemma \ref{lGenericwithoddrank}, we may and will henceforth assume that $E$ is $2$-generic and
\begin{align}
\label{eNegRoot}
w(E/K) = -1.
\end{align}
Take $T$ to be any finite set of places including all $2$-adic places, $3$-adic places, all infinite places and all places of bad reduction of $E$. For $t \in K^\ast/K^{\ast 2}$ write $v_1, \dots, v_r$ for the places that are outside of $T$ and ramified in $\psi_t$. We let $\pi_i(t)$ be the unique element of $K_{v_i}^\ast/K_{v_i}^{\ast 2}$ satisfying $\res_{v_i} t = \pi_i(t)$, and we write $\boldsymbol{\pi}(t) = (\pi_i(t))_{1 \leq i \leq r}$. With this notation set, we define the Selmer structure $\mathcal{L}_{i, t}$ to be $\mathcal{L}_{i, \boldsymbol{\pi}(t)}$. 

We will implicitly identify $\frac{1}{2} \Z/\Z \cong \mathbb{F}_2$ throughout this section.

\begin{mydef}
We say that an element $t \in K^\ast/K^{\ast 2}$ is a suitable twist if
\begin{enumerate}
\item[$(P1)$] $\psi_t$ is locally trivial at all places in $T$;
\item[$(P2)$] there exists $\kappa \in K^\ast/K^{\ast 2}$ and prime elements $q_1, q_2, q_3, q_4$ with the following properties
\begin{itemize}
\item we have
$$
\psi_t = \psi_\kappa + \psi_{q_1} + \psi_{q_2} + \psi_{q_3} + \psi_{q_4},
$$
\item denoting by $\{\mathfrak{p}_1, \dots, \mathfrak{p}_s\}$ the ramified primes in $\psi_\kappa$ outside of $T$, the elements $q_1, q_2, q_3, q_4$ are pairwise coprime and coprime to $\mathfrak{p}_1, \dots, \mathfrak{p}_s$ and $T$. We will list the ramified places of $\psi_t$ outside of $T$ as
$$
(\mathfrak{p}_1, \ldots, \mathfrak{p}_s, (q_1), \ldots, (q_4)) = (v_1, \ldots, v_r);
$$
\end{itemize}
\item[$(P3)$] writing $T' := T \cup \{\mathfrak{p}_1, \dots, \mathfrak{p}_s\}$, there exists a basis
$$
(z_1, z_2), \quad (z_3, z_4), \quad (z_5, z_6), \quad (z_7, z_8), \quad (z_9, z_{10})
$$
of $\mathrm{Sel}_{\mathcal{L}_{s, t}}(G_K, E[2])$ such that
\begin{align}
&\sum_{v \in T'} \inv_v(z_i \cup \psi_{q_1}) = 1 \quad \textup{ for } i \in \{1, \dots, 3\} \label{eP1} \\
&\sum_{v \in T'} \inv_v(z_i \cup \psi_{q_1}) = 0 \quad \textup{ for } i \in \{4, \dots, 10\} \label{eP2} \\
&\sum_{v \in T'} \inv_v(z_i \cup \psi_{q_2}) = 1 \quad \textup{ for } i \in \{6, 7\} \label{eP3} \\
&\sum_{v \in T'} \inv_v(z_i \cup \psi_{q_2}) = 0 \quad \textup{ for } i \in \{5, 8, 9, 10\} \label{eP4} \\
&\sum_{v \in T'} \inv_v(z_i \cup \psi_{q_3}) = 1 \quad \textup{ for } i \in \{9, 10\} \label{eP5};
\end{align}
\item[$(P4)$] we have $\mathrm{rk} \, E^t(K) > 0$.
\end{enumerate}
\end{mydef}

The conditions in $(P3)$ have been set up in such a way that, upon applying Lemma \ref{lSelmerChange} and Hilbert reciprocity, the Selmer rank will decrease from $5$ to $1$. The proof of our next result is similar to \cite[Theorem 5.2]{KP}.

\begin{theorem}
\label{tInfSuitable}
Suppose that there are infinitely many suitable $t$. Then Theorem \ref{tMain} holds.
\end{theorem}

\begin{proof}
Let $t$ be suitable. We claim that
\begin{align}
\label{e3Claim}
\dim_{\FF_2} \mathrm{Sel}^2(E^t/K) = 3.
\end{align}
We first explain why the claim implies the theorem. To this end, we note that $E^t(K)[2] \cong \mathbb{F}_2^2$ and hence we have the inequality
$$
2 + \mathrm{rk} \, E^t(K) \leq \dim_{\FF_2} \mathrm{Sel}^2(E^t/K).
$$
Then $(P4)$ and the claim \eqref{e3Claim} force $\mathrm{rk} \, E^t(K) = 1$, as desired.

Hence it remains to establish \eqref{e3Claim}. It is a consequence of $(P1)$ and $(P2)$ that the ramified primes of $\psi_t$ are precisely $\{\mathfrak{p}_1, \dots, \mathfrak{p}_s, (q_1), (q_2), (q_3), (q_4)\}$. We now apply Lemma \ref{lFinalStep}. Note that the hypothesis \eqref{eLocalT} of this lemma is satisfied because of $(P1)$. Then the claim \eqref{e3Claim} is equivalent to
\begin{align}
\label{eSelClaim}
\dim_{\FF_2} \mathrm{Sel}_{\mathcal{L}_{s + 3, t}}(G_K, E[2]) = 1,
\end{align}
which we prove now.

In order to establish equation \eqref{eSelClaim}, we shall make the intermediate claims that
\begin{align}
\label{eIClaim1}
\mathrm{Sel}_{\mathcal{L}_{s + 1, t}}(G_K, E[2]) = \langle (z_5, z_6), (z_7, z_8), (z_9, z_{10}) \rangle
\end{align}
and that
\begin{align}
\label{eIClaim2}
\mathrm{Sel}_{\mathcal{L}_{s + 2, t}}(G_K, E[2]) = \langle (z_9, z_{10}) \rangle.
\end{align}
We start by proving equation \eqref{eIClaim1}. Hilbert reciprocity gives
$$
\sum_{v \in \Omega_K} \inv_v(z_i \cup \psi_{q_1}) = 0
$$
for every $i \in \{1, \dots, 10\}$. Using $(P3)$ (more specifically equations \eqref{eP1} and \eqref{eP2}) and that the cup product of unramified characters vanishes, we obtain that
\begin{alignat}{2}
&(z_1(\Frob_{q_1}), z_2(\Frob_{q_1})) &&= (1, 1) \label{e1Red} \\ 
&(z_3(\Frob_{q_1}), z_4(\Frob_{q_1})) &&= (1, 0) \label{e2Red} \\ 
&(z_5(\Frob_{q_1}), z_6(\Frob_{q_1})) &&= (0, 0) \label{e3Red} \\ 
&(z_7(\Frob_{q_1}), z_8(\Frob_{q_1})) &&= (0, 0) \label{e4Red} \\ 
&(z_9(\Frob_{q_1}), z_{10}(\Frob_{q_1})) &&= (0, 0). \label{e5Red}
\end{alignat}
In particular, we deduce that $\dim_{\FF_2} \res_{(q_1)}(\mathrm{Sel}_{\mathcal{L}_{s, t}}(G_K, E[2])) = 2$, and hence we obtain
$$
\dim_{\FF_2} \mathrm{Sel}_{\mathcal{L}_{s + 1, t}}(G_K, E[2]) = 3
$$
from Lemma \ref{lSelmerChange}. Inspecting the equations \eqref{e1Red}, \eqref{e2Red}, \eqref{e3Red}, \eqref{e4Red}, \eqref{e5Red} one more time, we see that $(z_5, z_6), (z_7, z_8), (z_9, z_{10})$ are clearly in $\mathrm{Sel}_{\mathcal{L}_{s + 1, t}}(G_K, E[2])$, but they are also linearly independent by $(P3)$ (as they are part of a basis). This proves equation \eqref{eIClaim1}.

The proof of the second intermediate claim \eqref{eIClaim2} proceeds among the same lines as the proof of the first intermediate claim \eqref{eIClaim1}, and is omitted. Then it remains to show that equation \eqref{eIClaim2} implies \eqref{eSelClaim}. Arguing as before and using equation \eqref{eP5}, we see that
$$
\dim_{\FF_2} \res_{(q_3)}(\mathrm{Sel}_{\mathcal{L}_{s + 2, t}}(G_K, E[2])) = 1.
$$
Therefore we must be in the third case of Lemma \ref{lSelmerChange}, and the theorem follows.
\end{proof}

Our next reduction step constructs infinitely many suitable $t$ assuming that $\kappa$ has a certain convenient shape. We will establish the existence of such $\kappa$ in Theorem \ref{tFinal}. 

Using the fact that the elliptic curve $E$ is $2$-generic, we fix principal primes $w_1, \dots, w_5 \in T$ with generators $\lambda_1, \dots, \lambda_5$ such that
\begin{itemize}
\item $\lambda_1, \dots, \lambda_5$ are totally positive, satisfy $\lambda_i \equiv 1 \bmod 8O_K$ and satisfy $|O_K/\lambda_i| > 5$,
\item $w_1(\alpha) \equiv w_5(\alpha) \equiv 1 \bmod 2$, and $\beta$ and $\gamma$ are invertible squares modulo $w_1$ and $w_5$, 
\item $w_2(\beta) \equiv w_4(\beta) \equiv 1 \bmod 2$, and $\alpha$ and $\gamma$ are invertible squares modulo $w_2$ and $w_4$,
\item $w_3(\gamma) \equiv 1 \bmod 2$ and $\alpha$ and $\beta$ are invertible squares modulo $w_3$. 
\end{itemize}
In lieu of the presence of these elements, we will denote, for the rest of the paper, Selmer classes multiplicatively (as pairs of elements in $K^{\ast}/K^{\ast 2}$, rather than a pair of additive characters), as this will significantly lighten up the notation.  
Then we have by Lemma \ref{lSelmeratBad}
\begin{align}
&\delta(E(K_{w_1})) = \langle (\lambda_1, \lambda_1), (\epsilon_{w_1}, \epsilon_{w_1}) \rangle \label{eDeltaw1} \\
&\delta(E(K_{w_2})) = \langle (\lambda_2, 1), (\epsilon_{w_2}, 1) \rangle \label{eDeltaw2} \\
&\delta(E(K_{w_3})) = \langle (1, \lambda_3), (1, \epsilon_{w_3}) \rangle \label{eDeltaw3} \\
&\delta(E(K_{w_4})) = \langle (\lambda_4, 1), (\epsilon_{w_4}, 1) \rangle \label{eDeltaw4} \\
&\delta(E(K_{w_5})) = \langle (\lambda_5, \lambda_5), (\epsilon_{w_5}, \epsilon_{w_5}) \rangle, \label{eDeltaw5}
\end{align}
where we have written $\epsilon_{w_i}$ for the unique non-trivial unramified class.

\begin{mydef} 
\label{def: auxiliary triple}
We call $\kappa \in K^\ast/K^{\ast 2}$ an auxiliary twist if the following conditions simultaneously hold
\begin{enumerate}
\item[$(K1)$] $\psi_\kappa$ is locally trivial at $T$ and unramified at all places dividing $\alpha \beta \gamma$;
\item[$(K2)$] writing $\mathfrak{p}_1, \dots, \mathfrak{p}_s$ for the ramification locus of $\psi_\kappa$, there exists $\boldsymbol{\pi} = (\pi_i)_{1 \leq i \leq s}$ such that $\mathrm{Sel}_{\mathcal{L}_{s, \boldsymbol{\pi}}}(G_K, E[2])$ has a basis
$$
(z_1, z_2), (z_3, z_4), (z_5, z_6), (z_7, z_8), (z_9, z_{10})
$$
such that the valuations of $z_i$ (modulo $2$) at $w_1, \dots, w_5$ are given by the table
\begin{align}
\label{eInfMat1}
\begin{pmatrix}
& w_1 & w_2 & w_3 & w_4 & w_5 \\
w_i(\text{im}(\delta)) & (1, 1) & (1, 0) & (0, 1) & (1, 0) & (1, 1) \\
z_1    & 1 & 0 & 0 & 0 & 0 \\
z_2    & 1 & 0 & 0 & 0 & 0 \\
z_3    & 0 & 1 & 0 & 0 & 0 \\
z_4    & 0 & 0 & 0 & 0 & 0 \\
z_5    & 0 & 0 & 0 & 0 & 0 \\
z_6    & 0 & 0 & 1 & 0 & 0 \\
z_7    & 0 & 0 & 0 & 1 & 0 \\
z_8    & 0 & 0 & 0 & 0 & 0 \\
z_9    & 0 & 0 & 0 & 0 & 1 \\
z_{10} & 0 & 0 & 0 & 0 & 1 
\end{pmatrix}
.
\end{align}
\end{enumerate}
\end{mydef}

\noindent Write $N$ for the product of odd places in $T - \{w_1, \dots, w_5\}$. Our next theorem is the technical heart of our argument, and uses additive combinatorics as a key input.

\begin{theorem} 
\label{thm: auxiliary implies suitable}
Assume that there exists an auxiliary twist $\kappa \in K^\ast/K^{\ast 2}$. Then there exist infinitely many suitable $t \in K^\ast/K^{\ast 2}$.
\end{theorem}

\begin{proof}
The proof will proceed in several steps. Our eventual goal will be to apply the additive combinatorics result in \cite[Theorem A.8]{KP} (which is based on work of Kai \cite{Kai}). This result requires linear forms and a convex region $\Omega$ as an input. We start by constructing the linear forms. After this, we construct the convex region $\Omega$. In order to show that the main term in \cite[Theorem A.8]{KP} dominates, we must check that the linear forms are admissible (i.e.~there is no congruence obstruction to representing prime elements) and that the volume of $\Omega$ grows sufficiently fast. After we have done this, we apply the aforementioned additive combinatorics result, and then proceed to check $(P1)$, $(P2)$, $(P3)$ and $(P4)$.

\subsubsection*{Construction of linear forms}
By strong approximation and by our assumption $(K1)$, we may choose an integral representative for $\kappa$ that is coprime to $\alpha \beta \gamma$, totally positive and satisfies
$$
\kappa \equiv 1 \bmod 8N \lambda_1 \cdots \lambda_5.
$$
Take $\rho$ to be a generator of $8 N^{h_K}$. We invoke $(K2)$ to get primes $\mathfrak{p}_1, \dots, \mathfrak{p}_s$ that form the ramification locus of $\psi_\kappa$ and a choice of $\boldsymbol{\pi} = (\pi_i)_{1 \leq i \leq s}$ such that $\mathrm{Sel}_{\mathcal{L}_{s, \boldsymbol{\pi}}}(G_K, E[2])$ has a basis
$$
(z_1, z_2), \quad (z_3, z_4), \quad (z_5, z_6), \quad (z_7, z_8), \quad (z_9, z_{10})
$$
with the parities of the valuations at $w_1, \dots, w_5$ as given by the table \eqref{eInfMat1}.

By strong approximation, there exists $\lambda \in O_K$ coprime to $\kappa$ such that
\begin{align}
\label{elambdachoice}
&\res_{\mathfrak{p}_i}(\kappa \lambda) = \pi_i \quad \text{for all } 1 \leq i \leq s \\
&\lambda \equiv 1 \bmod \rho. \label{elambdachoice2}
\end{align}
By strong approximation and Lemma \ref{lfinitefields}, there exist $\mu_1, \mu_2 \in O_K$ such that the Legendre symbols of the following four linear forms are given by the table
\begin{align}
\label{eInfMat2}
\begin{pmatrix}
& \left(\frac{\cdot}{\lambda_1}\right) & \left(\frac{\cdot}{\lambda_2}\right) & \left(\frac{\cdot}{\lambda_3}\right) & \left(\frac{\cdot}{\lambda_4}\right) & \left(\frac{\cdot}{\lambda_5}\right) \\
\rho^2 \kappa \mu_1 - a_1 \rho^4 \kappa^2 \mu_2 - a_1 \rho^2 \kappa \lambda + 1 & -1 & -1 & +1 & +1 & +1 \\
\rho^2 \kappa \mu_1 - a_2 \rho^4 \kappa^2 \mu_2 - a_2 \rho^2 \kappa \lambda + 1 & -1 & +1 & -1 & -1 & +1 \\
\rho^2 \kappa \mu_1 - a_3 \rho^4 \kappa^2 \mu_2 - a_3 \rho^2 \kappa \lambda + 1 & +1 & -1 & -1 & +1 & -1 \\
\rho^2 \kappa \mu_2 + \lambda & +1 & +1 & +1 & -1 & -1
\end{pmatrix}
.
\end{align}
We are able to apply Lemma \ref{lfinitefields}, which involves only three linear constraints, because two linear demands coincide for each $w_i$ (the first two if $w_i(\alpha) > 0$, the last two if $w_i(\gamma) > 0$, the first and the third if $w_i(\beta) > 0$). This phenomenon can be viewed as a non-archimedean analogue of the argument on bottom of \cite[p.~20]{KP}. 

We define $m := \lambda_1 \cdots \lambda_5$ and we define four affine linear forms $L_1, L_2, L_3, L_4 \in O_K[X, Y]$ 
\begin{align*}
L_1(X, Y) &:= \rho^2 \kappa \left(m X + \mu_1\right) - a_1 \rho^2 \kappa \left(\rho^2 \kappa \left(m Y+\mu_2\right) + \lambda\right) + 1 \\
L_2(X, Y) &:= \rho^2 \kappa \left(m X + \mu_1\right) - a_2 \rho^2 \kappa \left(\rho^2 \kappa \left(m Y + \mu_2\right) + \lambda\right) + 1 \\
L_3(X, Y) &:= \rho^2 \kappa \left(m X + \mu_1\right) - a_3 \rho^2 \kappa \left(\rho^2 \kappa \left(m Y + \mu_2\right) + \lambda\right) + 1 \\
L_4(X, Y) &:= \rho^2 \kappa \left(m Y + \mu_2\right) + \lambda.
\end{align*}
Recall that the linear forms $L_1, L_2, L_3, L_4$ are called admissible if for every finite prime ideal $\mathfrak{p}$, there exists $u, v \in O_K$ such that $L_1(u, v) L_2(u, v) L_3(u, v) L_4(u, v) \not \equiv 0 \bmod \mathfrak{p}$.

\begin{lemma}
\label{lAdmissible}
The linear forms $L_1, L_2, L_3, L_4$ are admissible.
\end{lemma}

\begin{proof}
We distinguish four cases. Firstly, let $\mathfrak{p}$ be a prime ideal dividing $\rho$. In that case we take $u = v = 0$, and we have $L_i(u, v) \equiv 1 \bmod \mathfrak{p}$ for $i \in \{1, 2, 3\}$ and $L_4(u, v) \equiv \lambda \equiv 1 \bmod \mathfrak{p}$ by equation \eqref{elambdachoice2}.

Secondly, suppose that $\mathfrak{p}$ divides $m$. Then we take $u = v = 0$, so
\begin{align*}
&L_i(u, v) \equiv \rho^2 \kappa \mu_1 - a_i \rho^4 \kappa^2 \mu_2 - a_i \rho^2 \kappa \lambda + 1 \bmod \mathfrak{p} \quad \quad \text{ for } i \in \{1, 2, 3\}, \\
&L_4(u, v) \equiv \rho^2 \kappa \mu_2 + \lambda \bmod \mathfrak{p}.
\end{align*}
In all cases, we see that $L_i(u, v)$ is invertible modulo $\mathfrak{p}$ by construction of $\mu_1$ and $\mu_2$. 

Thirdly, let $\mathfrak{p}$ be a prime dividing $\kappa$. We again take $u = v = 0$, and then we have
\begin{align*}
&L_i(u, v) \equiv 1 \bmod \mathfrak{p} \quad \quad \text{ for } i \in \{1, 2, 3\}, \\
&L_4(u, v) \equiv \lambda \bmod \mathfrak{p}.
\end{align*}
Recalling that $\lambda$ is coprime to $\mathfrak{p}$ by construction, this case is now also covered.

Finally, let $\mathfrak{p}$ be any place not dividing $\rho m \kappa$. Observe that this implies $\mathfrak{p} \nmid 6$. Now pick any $v$ such that $\rho^2 \kappa m v + \rho^2 \kappa \mu_2 + \lambda \not \equiv 0 \bmod \mathfrak{p}$. Since $\mathfrak{p} \nmid 6$, we have $|O_K/\mathfrak{p}| \geq 5$, and hence there exists $u$ such that
$$
u \bmod \mathfrak{p} \not \in \left\{\frac{-\rho^2 \kappa \mu_1 - 1 + a_i \rho^2 \kappa \left(\rho^2 \kappa \left(m v + \mu_2\right) + \lambda\right)}{\rho^2 \kappa m} \bmod \mathfrak{p} : i \in \{1, 2, 3\}\right\}.
$$
With these choices, we directly verify that $L_1(u, v) L_2(u, v) L_3(u, v) L_4(u, v) \not \equiv 0 \bmod \mathfrak{p}$ ending the proof of the lemma.
\end{proof}

\subsubsection*{Construction of a convex region}
We will now construct a region $\Omega$ to which we apply Kai's theorem as stated in \cite[Theorem A.8]{KP}. Fix an integral basis $\omega_1, \dots, \omega_n$ of $O_K$, and define the affine linear form $\widetilde{L_i}: \Z^{2n} \rightarrow O_K$ by
$$
(a_1, \dots, a_n, b_1, \dots, b_n) \mapsto \widetilde{L_i}(a_1 \omega_1 + \dots + a_n \omega_n, b_1 \omega_1 + \dots + b_n \omega_n).
$$
For each real embedding $\sigma: K \rightarrow \mathbb{R}$, this gives an affine linear map $\sigma \circ \widetilde{L_i}: \Z^{2n} \rightarrow \R$, and this gives an affine linear map $\sigma \circ \widetilde{L_i}: \R^{2n} \rightarrow \R$ by viewing $\Z^{2n}$ inside $\R^{2n}$ in the obvious way. Take
$$
\Omega := \{(\mathbf{u}, \mathbf{v}) \in \R^{2n} : \sigma(\widetilde{L_i}(\mathbf{u}, \mathbf{v})) > 0 \text{ for all real places } \sigma, \text{ for all } i \in \{1, 2, 3, 4\}\}.
$$

\begin{lemma}
\label{lOmegaVol}
There exists $C > 0$ such that
$$
|\Omega \cap [-B, B]^{2n}| \sim C B^{2n}
$$
as $B \rightarrow \infty$.
\end{lemma}

\begin{proof}
Let $\sigma: K \rightarrow \mathbb{R}$ be a real embedding. We will now rewrite the four inequalities
$$
\sigma(\widetilde{L_i}(\mathbf{u}, \mathbf{v})) > 0 \text{ for all } i \in \{1, 2, 3, 4\}
$$
as two inequalities. Define $i_{1, \sigma}$ to be the unique index $i$ where $\sigma(a_i)$ is maximal and define $i_{2, \sigma} := 4$. Then the four inequalities are equivalent to
$$
\sigma(\widetilde{L_{i_{1, \sigma}}}(\mathbf{u}, \mathbf{v})) > 0 \quad \quad \sigma(\widetilde{L_{i_{2, \sigma}}}(\mathbf{u}, \mathbf{v})) > 0.
$$
Define $M_{1, \sigma}$ and $M_{2, \sigma}$ to be the homogeneous linear forms corresponding to $\widetilde{L_{i_{1, \sigma}}}$ and $\widetilde{L_{i_{2, \sigma}}}$. Our lemma is a consequence of \cite[Lemma 5.6]{KP} provided that we check that the system of linear forms $\{M_{1, \sigma} : \sigma \text{ real}\} \cup \{M_{2, \sigma} : \sigma \text{ real}\}$ are linearly independent over $\R$. Writing $\widetilde{\sigma}: \mathbb{R}^n \rightarrow \mathbb{R}$ for the extension of $\sigma$, the distinctness of the $a_i$ implies
$$
M_{1, \sigma}(\mathbf{u}, \mathbf{v}) = M_{2, \sigma}(\mathbf{u}, \mathbf{v}) = 0 \Longleftrightarrow \widetilde{\sigma}(\mathbf{u}) = \widetilde{\sigma}(\mathbf{v}) = 0.
$$
Hence the linear independence follows from the linear independence of $\widetilde{\sigma}$.
\end{proof}

We now apply \cite[Theorem A.8]{KP}. From the asymptotic in \cite[Theorem A.8]{KP}, by Lemma \ref{lAdmissible} and Lemma \ref{lOmegaVol}, this result provides infinitely many quadruples $(q_{1, j}, q_{2, j}, q_{3, j}, q_{4, j})_{j \geq 1}$ attained by the linear forms $L_1, L_2, L_3, L_4$, where the $q_{i, j}$ are coprime to $\mathfrak{p}_1, \dots, \mathfrak{p}_s$ and $T$ and where the ideals $(q_{i, j})$ are all distinct.

Now fix one such quadruple $(q_1, q_2, q_3, q_4)$, and define
\begin{equation}
\label{eKappat}
t := \kappa q_1 q_2 q_3 q_4.
\end{equation}
It remains to check that outside finitely many such $q_1, q_2, q_3, q_4$, the properties $(P1)$, $(P2)$, $(P3)$ and $(P4)$ hold.

\subsubsection*{Verification of $(P1)$ and $(P2)$}
As for $(P1)$, note that
$$
q_1 \equiv q_2 \equiv q_3 \equiv q_4 \equiv 1 \bmod 8N.
$$
Moreover, for the places $w_1, \dots, w_5$ an inspection of the table \eqref{eInfMat2} reveals for all $i \in \{1, \dots, 5\}$ the identity
$$
q_1 q_2 q_3 q_4 \equiv \square \bmod \lambda_i.
$$
Since $\psi_\kappa$ is locally trivial at all places in $T$, so is 
$$
\psi_t = \psi_\kappa + \psi_{q_1} + \psi_{q_2} + \psi_{q_3} + \psi_{q_4}.
$$
As for $(P2)$, this follows from the fact that we chose $q_1, q_2, q_3, q_4$ coprime to $\mathfrak{p}_1, \dots, \mathfrak{p}_s$ and $T$.

\subsubsection*{Verification of $(P3)$}
For $(P3)$, recall that we fixed a choice of $\boldsymbol{\pi} = (\pi_i)_{1 \leq i \leq s}$ satisfying $(K2)$. By our choice of linear forms, we have the congruences
\begin{align}
\label{eqCong}
q_1 \equiv q_2 \equiv q_3 \equiv 1 \bmod 8N \kappa, \quad \quad q_4 \equiv \lambda \bmod 8N \kappa.
\end{align}
Therefore our choice of $\lambda$, see equation \eqref{elambdachoice}, implies that 
$$
\pi_i(t) = \pi_i(\kappa q_1 q_2 q_3 q_4) = \res_{\mathfrak{p}_i}(\kappa q_1 q_2 q_3 q_4) = \res_{\mathfrak{p}_i}(\kappa \lambda) = \pi_i
$$
for all $1 \leq i \leq s$. Hence we have
$$
\mathrm{Sel}_{\mathcal{L}_{s, t}}(G_K, E[2]) = \mathrm{Sel}_{\mathcal{L}_{s, \boldsymbol{\pi}}}(G_K, E[2]).
$$
We will now derive $(P3)$ as a consequence of the matrices \eqref{eInfMat1} and \eqref{eInfMat2}. Since $q_1, q_2, q_3$ are totally positive and satisfy the congruence $q_i \equiv 1 \bmod 8N \kappa$ by \eqref{eqCong}, we observe that
$$
\sum_{v \in T'} \inv_v(z_i \cup q_j) = \sum_{v \in \{w_1, w_2, w_3, w_4, w_5\}} \inv_v(z_i \cup q_j),
$$
which we will now use to verify \eqref{eP1}, \eqref{eP2}, \eqref{eP3}, \eqref{eP4}, \eqref{eP5}.

\paragraph{Proof of equation \eqref{eP1} and \eqref{eP2}}
Since $\psi_{q_1}$ is locally trivial at $w_3, w_4, w_5$ and a non-trivial unit at $w_1, w_2$ by \eqref{eInfMat2}, we get
$$
\sum_{v \in \{w_1, w_2, w_3, w_4, w_5\}} \inv_v(z_i \cup q_1)  = w_1(z_i) + w_2(z_i).
$$
The result follows from an inspection of the table \eqref{eInfMat1}.

\paragraph{Proof of equation \eqref{eP3} and \eqref{eP4}}
Since $\psi_{q_2}$ is locally trivial at $w_3, w_4, w_5$ and a non-trivial unit at $w_1, w_2$ by \eqref{eInfMat2}, we obtain
$$
\sum_{v \in \{w_1, w_2, w_3, w_4, w_5\}} \inv_v(z_i \cup q_2)  = w_1(z_i) + w_3(z_i) + w_4(z_i).
$$
The result follows from an inspection of the table \eqref{eInfMat1}.

\paragraph{Proof of equation \eqref{eP5}.} 
Since $\psi_{q_3}$ is locally trivial at $w_1, w_4$ and a non-trivial unit at $w_2, w_3, w_5$ by \eqref{eInfMat2}, we deduce
$$
\sum_{v \in \{w_1, w_2, w_3, w_4, w_5\}} \inv_v(z_i \cup q_3)  = w_2(z_i) + w_3(z_i) + w_5(z_i).
$$
The result follows from an inspection of the table \eqref{eInfMat1}.

\subsubsection*{Verification of $(P4)$}
For $(P4)$, we take $(x, y) \in O_K^2$ with $q_{i, j} = L_i(x, y)$ for $i \in \{1, 2, 3, 4\}$. We define
$$
c := \rho^2 \kappa \left(m x + \mu_1\right) + 1, \quad \quad d := \rho^2 \kappa (\rho^2 \kappa \left(m y + \mu_2\right) + \lambda).
$$
Then we have by construction
$$
L_i(x, y) = c - a_i d, \quad \quad d = \rho^2 \kappa L_4(x, y).
$$
Recalling equation \eqref{eKappat}, we see that
$$
t = \kappa q_{1, j} q_{2, j} q_{3, j} q_{4, j} = \rho^{-2} d (c - a_1 d) (c - a_2 d) (c - a_3 d),
$$
and thus the quadratic twist $E^t$ has the rational point $(c/d, \rho/d^2)$. Hence, outside of finitely many quadruples $(q_{1, j}, q_{2, j}, q_{3, j}, q_{4, j})_{i \geq 1}$, property $(P4)$ follows from Lemma \ref{lTorsion}.
\end{proof}

\subsection{Construction of an auxiliary twist}
In order to complete the proof of Theorem \ref{tMain}, it remains to prove that an auxiliary twist exists. In our later arguments, it will be convenient to gain some partial control on a slightly modified Selmer group, relaxing our local conditions at the places in $T$. Given $v_1, \dots, v_i \not \in T$, we introduce the Selmer structure $\mathcal{L}_{i, \boldsymbol{\pi}}^{\text{ray}} = (\mathcal{L}_{i, \boldsymbol{\pi}, v}^{\text{ray}})_{v \in \Omega_K}$ given by
$$
\mathcal{L}_{i, \boldsymbol{\pi}, v}^{\text{ray}} = 
\begin{cases}
H^1(G_{K_v} , E[2]) &\text{if } v \in T, \\
\langle (\alpha \beta, \pi_j \alpha), (-\pi_j \alpha, - \alpha \gamma) \rangle &\text{if } v = v_j \text{ for some } j \in \{1, \dots, i\}, \\
H^1_{\text{nr}}(G_{K_v}, E[2]) &\text{otherwise.}
\end{cases}
$$
We define 
$$
\text{Diag}(i, \pi) := \dim_{\mathbb{F}_2} V_1(i, \boldsymbol{\pi}) + \dim_{\mathbb{F}_2} V_2(i, \boldsymbol{\pi}) + \dim_{\mathbb{F}_2} V_3(i, \boldsymbol{\pi}),
$$
where $V_1(i, \boldsymbol{\pi}) , V_2(i, \boldsymbol{\pi}), V_3(i, \boldsymbol{\pi})$ are the subspace of $\mathrm{Sel}_{\mathcal{L}_{i, \boldsymbol{\pi}}^{\text{ray}}}(G_K , E[2])$ consisting respectively of elements of the form $(x, 1)$, $(1, x)$ and $(x, x)$. 

Finally, we will frequently use a basic lemma on Legendre symbols, so we state it now.

\begin{lemma}
\label{lFlip}
Let $x, y \in O_K$ be odd. Suppose that $x$ is a square locally at all $2$-adic places and all infinite places. Then
$$
\left(\frac{x}{y}\right) = \left(\frac{y}{x}\right).
$$
\end{lemma}

\begin{proof}
Immediate from Hilbert reciprocity.
\end{proof}

\begin{theorem}
\label{tFinal}
There exists an auxiliary twist $\kappa \in K^\ast/K^{\ast 2}$.
\end{theorem}

\begin{proof}
We will construct a sequence of prime ideals $\mathfrak{p}_1, \ldots, \mathfrak{p}_s$ for some sufficiently large integer $s \geq 1$, and an associated sequence of local uniformizers $\pi_1 , \dots, \pi_s$. Once these choices have been made, we will declare our auxiliary twist $\kappa$, and then show that it satisfies properties $(K1)$ and $(K2)$. We will now broadly overview the various steps.

In the first part, we make $\mathrm{Sel}_{\mathcal{L}_{i, \boldsymbol{\pi}}}(G_K , E[2])$ as small as possible. Concretely, this means
$$
\dim_{\FF_2} \mathrm{Sel}_{\mathcal{L}_{s - 17, \boldsymbol{\pi}}}(G_K , E[2]) = 1. 
$$
Unfortunately, we shall need even more from this space in our later arguments, and for this reason we also need control on several auxiliary spaces. Once we have obtained this control, we will use $17$ more prime ideals before declaring $\kappa$.

In the second part, we introduce six convenient auxiliary prime elements. These prime elements are set up so that they satisfy the local conditions in the matrix \eqref{eInfMat1} but are purposely chosen to be almost but not quite Selmer elements. Nevertheless, our Selmer group $\mathrm{Sel}_{\mathcal{L}_{s - 17, \boldsymbol{\pi}}}(G_K , E[2])$ may grow during this construction to a space $\mathrm{Sel}_{\mathcal{L}_{s - 11, \boldsymbol{\pi}}}(G_K , E[2])$ of dimension at most $13$. In the third part, we use six prime ideals and the control on our auxiliary spaces to remove any additional Selmer elements introduced by the six auxiliary prime elements. In particular, we ensure that
$$
\dim_{\FF_2} \mathrm{Sel}_{\mathcal{L}_{s - 5, \boldsymbol{\pi}}}(G_K , E[2]) = 1. 
$$
In the fourth part, we now introduce three further prime elements. These prime elements are set up precisely so that they give six linearly independent Selmer elements once combined appropriately with the previous six prime elements (thus ensuring $\dim_{\FF_2} \mathrm{Sel}_{\mathcal{L}_{s - 2, \boldsymbol{\pi}}}(G_K , E[2]) = 7$) and moreover satisfy the conditions in \eqref{eInfMat1}. In the fifth part, we check that these elements are indeed Selmer elements by a careful examination of local conditions. In the sixth part, we introduce two more prime ideals and declare $\kappa$. The first prime ideal removes the extra Selmer classes (in order to obtain a $5$-dimensional space), while the second prime ideal principalizes the product of all prime ideals so that we can declare $\kappa$. Finally, we check $(K1)$ and $(K2)$ in the seventh part.

\subsubsection*{Decreasing ranks}
Define $S$ to be the set of places $T$ together with all places dividing $\alpha \beta \gamma$. To start, we choose any sequence of prime ideals $\mathfrak{p}_1, \ldots, \mathfrak{p}_{i_0} \not \in S$ and any sequence of local uniformizers $\pi_1, \dots, \pi_{i_0}$ such that the Frobenius elements $\Frob_{\mathfrak{p}_1}, \dots, \Frob_{\mathfrak{p}_{i_0}}$ generate the Galois group $\Gal(K(S)/K)$, where $K(S)$ denotes the largest multiquadratic extension of $K$ that is unramified outside of $S$. Then, by examining the local conditions at the primes in $\mathfrak{p}_1, \dots, \mathfrak{p}_{i_0}$, we have for all $i \geq i_0$ the implication
\begin{equation}
\label{eStartingRay}
(y_1 , y_2) \in \mathrm{Sel}_{\mathcal{L}_{i, \boldsymbol{\pi}}}(G_K , E[2]) \text{ and } y_1, y_2 \text{ unramified outside } S \Longrightarrow y_1 = y_2 = 1
\end{equation}
We will use throughout that the dimension of $\mathrm{Sel}_{\mathcal{L}_{i, \boldsymbol{\pi}}}(G_K, E[2])$ is odd; this follows upon combining equation \eqref{eNegRoot}, Lemma \ref{lRootNumber} and Lemma \ref{lSelmerChange} with $\mathrm{Sel}_{\mathcal{L}_{0, \boldsymbol{\pi}}}(G_K, E[2]) = \mathrm{Sel}^2(E/K)$.

We claim that there exists an integer $r' \geq 1$, a sequence of prime ideals $\mathfrak{p}_1, \ldots, \mathfrak{p}_{r'} \not \in S$ and local uniformizers $\pi_1, \ldots, \pi_{r'}$, with the following properties
\begin{align}
&\dim_{\FF_2} \mathrm{Sel}_{\mathcal{L}_{r', \boldsymbol{\pi}}}(G_K, E[2]) = 1, \label{edimension1} \\
&\text{Diag}(r', \boldsymbol{\pi}) \leq 1, \label{eRaygeneric}  \\
&V_2(r', \boldsymbol{\pi}) = 0. \label{eBasisSwap} 
\end{align}
We shall prove this in three steps, which we now overview. We will first construct $r_1$ such that equation \eqref{edimension1} holds. We will then preserve the validity of equation \eqref{edimension1} during our construction of an integer $r_2 \geq r_1$ with $\text{Diag}(r_2, \boldsymbol{\pi}) \leq 1$. We will then perform at most one additional step in case equation \eqref{eBasisSwap} fails: in this step we will ensure that \eqref{eBasisSwap} holds while preserving equations \eqref{edimension1} and \eqref{eRaygeneric}.

Let us begin by constructing $r_1$. To this end, we claim that there exists an integer $r_1 \geq 1$, a sequence of prime ideals $\mathfrak{p}_1, \ldots, \mathfrak{p}_{r_1}$ and local uniformizers $\pi_1, \ldots, \pi_{r_1}$ such that equation \eqref{edimension1} holds. This claim can be proven by modifying the argument ``Part 1: Reducing Selmer ranks'' on \cite[p.~23-24]{KP} as follows: the right hand side of $(S1)$ becomes $1$, even becomes odd on top of page $24$ and \cite[Corollary 3.7]{KP} is replaced by the fact that $E$ is $2$-generic. Indeed, since $E$ is $2$-generic, it is readily verified that $\alpha\beta, -\alpha\gamma, \beta\gamma$ are non-squares. With these changes, the proof on \cite[p.~23-24]{KP} can be followed verbatim. 

We shall now construct $r_2 \geq r_1$ in such a way that equation \eqref{eRaygeneric} holds, while preserving the truth of equation \eqref{edimension1}. To this end, we will explain how given a choice of prime ideals $\mathfrak{p}_1 , \dots , \mathfrak{p}_i$ and local uniformizers $\pi_1, \dots, \pi_i$ such that 
$$
\text{Diag}(i, \boldsymbol{\pi}) \geq 2, \quad \quad \dim_{\FF_2} \mathrm{Sel}_{\mathcal{L}_{i, \boldsymbol{\pi}}}(G_K, E[2]) = 1,
$$
it is always possible to choose two new prime ideals $\mathfrak{p}_{i + 1}, \mathfrak{p}_{i + 2}$ and local uniformizers $\pi_{i + 1}, \pi_{i + 2}$ in such a way that for at least one value of $j \in \{i + 1, i + 2\}$ we have that
\begin{align} 
&\dim_{\FF_2} \mathrm{Sel}_{\mathcal{L}_{j, \boldsymbol{\pi}}}(G_K, E[2]) = \dim_{\FF_2} \mathrm{Sel}_{\mathcal{L}_{i, \boldsymbol{\pi}}}(G_K, E[2]), \label{epreservingSelmer} \\
&\text{Diag}(j , \boldsymbol{\pi}) \leq \text{Diag}(i, \boldsymbol{\pi}) - 1. \label{eDecreaseDiag}
\end{align}
Once we have obtained this, we can iteratively apply this in order to find some integer $r_2 \leq r_1 + 2 \cdot \text{Diag}(r_1 , \boldsymbol{\pi})$ satisfying simultaneously equation \eqref{edimension1} and equation \eqref{eRaygeneric}. 

So we are left with proving the existence of prime ideals $\mathfrak{p}_{i + 1}, \mathfrak{p}_{i + 2}$ and local uniformizers $\pi_{i + 1}, \pi_{i + 2}$ satisfying equation \eqref{epreservingSelmer} and equation \eqref{eDecreaseDiag}. We distinguish two cases. Suppose first that at least two of the spaces $V_1(i, \boldsymbol{\pi}), V_2(i, \boldsymbol{\pi}), V_3(i, \boldsymbol{\pi})$ are non-zero, say $V_1(i, \boldsymbol{\pi}) \neq 0$ and $V_2(i, \boldsymbol{\pi}) \neq 0$ (the other cases can be treated similarly). Fix some non-trivial elements $(c , 1) \in V_1(i, \boldsymbol{\pi})$ and $(1 , d) \in V_2(i, \boldsymbol{\pi})$. In view of our choice of $\mathfrak{p}_1, \ldots, \mathfrak{p}_{i_0}$, see equation \eqref{eStartingRay}, the elements $c$ and $d$ are linearly independent from the span of $\{\alpha\beta, -\alpha\gamma, \beta\gamma\}$. Since $E$ is $2$-generic, the three elements $\alpha \beta$, $-\alpha \gamma$ and $\beta \gamma$ span a space of dimension at least $2$. Therefore there are two distinct elements $x, y \in \{\alpha\beta, -\alpha\gamma, \beta\gamma\}$ that are both different from $cd$. Observe that the only possible non-trivial linear relations between $c , d, x , y$ are then $cd = xy$ or $cd = 1$, but not both.

Therefore there exists a prime ideal $\mathfrak{p}_{i + 1}$ such that 
$$
\left(\left(\frac{c}{\mathfrak{p}_{i + 1}}\right), \left(\frac{d}{\mathfrak{p}_{i + 1}}\right), \left(\frac{x}{\mathfrak{p}_{i + 1}}\right), \left(\frac{y}{\mathfrak{p}_{i + 1}}\right)\right) = (-1, -1, -1, -1).
$$
Since 
$$
\dim_{\FF_2} \mathrm{Sel}_{\mathcal{L}_{i, \boldsymbol{\pi}}}(G_K, E[2]) = 1,
$$
we are always either in the first or the third case of Lemma \ref{lSelmerChange}. Hence we can always choose $\pi_{i + 1}$ in such a way that we are in the third case. This certainly enforces equation \eqref{epreservingSelmer}. It remains to show that equation \eqref{eDecreaseDiag} holds with these choices.

We will now distinguish three possibilities. Suppose that $\{x, y\} = \{\alpha\beta, -\alpha\gamma\}$. Then we notice that 
$$
V_1(i + 1 , \boldsymbol{\pi}) \subsetneq V_1(i, \boldsymbol{\pi}), \quad V_2(i + 1, \boldsymbol{\pi}) \subsetneq V_2(i, \boldsymbol{\pi}), \quad \dim_{\FF_2} \frac{V_3(i + 1, \boldsymbol{\pi})}{V_3(i +1 , \boldsymbol{\pi}) \cap V_3(i, \boldsymbol{\pi})} \leq 1.
$$
Therefore we have
$$
\text{Diag}(i + 1, \boldsymbol{\pi}) - \text{Diag}(i, \boldsymbol{\pi}) \leq -1-1+1=-1.
$$
Suppose now that $\{x, y\} = \{\alpha\beta, \beta\gamma\}$. Then we observe that 
$$
\dim_{\FF_2} \frac{V_1(i + 1, \boldsymbol{\pi})}{V_1(i + 1, \boldsymbol{\pi}) \cap V_1(i, \boldsymbol{\pi})} \leq 1, \quad V_2(i + 1, \boldsymbol{\pi}) \subsetneq V_2(i, \boldsymbol{\pi}), \quad V_3(i + 1, \boldsymbol{\pi}) \subseteq V_3(i, \boldsymbol{\pi})
$$
and moreover $V_1(i + 1 ,\boldsymbol{\pi}) \cap V_1(i, \boldsymbol{\pi}) \subsetneq V_1(i, \boldsymbol{\pi})$. Hence we conclude that
$$
\text{Diag}(i + 1, \boldsymbol{\pi}) - \text{Diag}(i, \boldsymbol{\pi}) \leq (1 - 1) -1 + 0 = 0 - 1 + 0 = -1.
$$
The remaining case $\{x, y\} = \{-\alpha\gamma, \beta\gamma\}$ is identical to the previous one upon permuting the roles of $1$ and $2$. This shows that equation \eqref{eDecreaseDiag} holds in this case. 

We now approach the second case, which is that $\text{Diag}(i, \boldsymbol{\pi}) \geq 2$ and precisely one of the three spaces $V_1(i, \boldsymbol{\pi}) , V_2(i, \boldsymbol{\pi}) , V_3(i, \boldsymbol{\pi})$ is non-zero. We will consider the case $V_1(i, \boldsymbol{\pi}) \neq 0$, the other cases being symmetrical. Let $(c, 1)$ be a non-zero element of this space. Since $c$ is linearly independent from $\{\alpha\beta, -\alpha\gamma, \beta\gamma\}$ by equation \eqref{eStartingRay}, and since any two distinct elements from $\{\alpha\beta, -\alpha\gamma, \beta\gamma\}$ are linearly independent, there exists a prime ideal $\mathfrak{p}_{i + 1}$ such that $c$, $\alpha\beta$ and $-\alpha\gamma$ are all non-squares modulo $\mathfrak{p}_{i + 1}$. Because
$$
\dim_{\FF_2} \mathrm{Sel}_{\mathcal{L}_{i, \boldsymbol{\pi}}}(G_K, E[2]) = 1,
$$
we are either in the first or the third case of Lemma \ref{lSelmerChange}. Hence we can always choose $\pi_{i + 1}$ in such a way that we are in the third case. This certainly enforces equation \eqref{epreservingSelmer}. We now claim that
$$
\text{Diag}(i + 1, \boldsymbol{\pi}) \leq \text{Diag}(i, \boldsymbol{\pi})
$$
and that if we have equality, then at least two of the spaces $V_1(i + 1, \boldsymbol{\pi}), V_2(i + 1, \boldsymbol{\pi}), V_3(i + 1, \boldsymbol{\pi})$ are non-zero. The claim implies equation \eqref{eDecreaseDiag}, since we are either done immediately (so $j := i + 1$) or we can apply the argument from the first case, in which case we take $j := i + 2$.

To prove the claim, observe that our choice of $\mathfrak{p}_{i + 1}$ implies that 
$$
V_1(i + 1 , \boldsymbol{\pi}) \subsetneq V_1(i, \boldsymbol{\pi}), \quad V_2(i + 1, \boldsymbol{\pi}) \subseteq V_2(i, \boldsymbol{\pi}), \quad \dim_{\FF_2} \frac{V_3(i + 1, \boldsymbol{\pi})}{V_3(i + 1, \boldsymbol{\pi}) \cap V_3(i, \boldsymbol{\pi})} \leq 1. 
$$
This already shows that
$$
\text{Diag}(i + 1 , \boldsymbol{\pi})-\text{Diag}(i, \boldsymbol{\pi}) \leq -1 + 0 + 1 = 0.
$$
Moreover, the only way to reach equality is that $V_3(i + 1 , \boldsymbol{\pi})$ has positive dimension. Since $\dim_{\FF_2} V_1(i, \boldsymbol{\pi}) \geq 2$ by assumption, we also know that $V_1(i + 1 , \boldsymbol{\pi})$, which is of codimension $1$ in $V_1(i, \boldsymbol{\pi})$, satisfies $\dim_{\FF_2} V_1(i + 1, \boldsymbol{\pi}) \geq 1$. This gives precisely the desired conclusion. 

We have so far constructed $\mathfrak{p}_1, \dots ,\mathfrak{p}_{r_2}$ and local uniformizers $\pi_1, \dots, \pi_{r_2}$ in such a way that equation \eqref{edimension1} and \eqref{eRaygeneric} hold. In order to enforce also equation \eqref{eBasisSwap}, we distinguish two cases. Suppose first that $V_2(r_2, \boldsymbol{\pi}) = 0$. Then we declare $r' := r_2$, and the equations \eqref{edimension1}, \eqref{eRaygeneric} and \eqref{eBasisSwap} all hold.

Suppose instead that $V_2(r_2, \boldsymbol{\pi}) \neq 0$. In virtue of equation \eqref{eRaygeneric}, it must be the case that $\dim_{\mathbb{F}_2} V_2(r_2 , \boldsymbol{\pi}) = 1$, while both $V_1(r_2, \boldsymbol{\pi})$ and $V_3(r_2, \boldsymbol{\pi})$ are trivial. Let $(1, z)$ be the unique non-trivial element of $V_2(r_2, \boldsymbol{\pi})$. From \eqref{eStartingRay} and genericity, we see that $z$, $\alpha\beta$ and $\beta\gamma$ are linearly independent. Thus there exists $\mathfrak{p}_{r_2 + 1}$ such that $z$, $\alpha\beta$ and $\beta\gamma$ are non-squares, while, as before, we choose $\pi_{r_2 + 1}$ in such a way that $n_{r_2} = 0$. Thanks to these choices we have that 
$$
\dim_{\FF_2} V_1(r_2 + 1 , \boldsymbol{\pi}) \leq 1, \quad V_2(r_2 + 1, \boldsymbol{\pi}) = 0, \quad V_3(r_2 + 1, \boldsymbol{\pi}) = 0. 
$$
We now take $r' := r_2 + 1$. With the above choices of prime ideals and local uniformizers, equations \eqref{edimension1}, \eqref{eRaygeneric} and \eqref{eBasisSwap} all simultaneously hold as claimed. 

\subsubsection*{Constructing prime elements, part 1}
Let us now fix a choice of $r'$, $\mathfrak{p}_1, \ldots, \mathfrak{p}_{r'}$ and $\pi_1, \ldots, \pi_{r'}$ satisfying equations \eqref{edimension1}, \eqref{eRaygeneric} and \eqref{eBasisSwap}. We also fix a principal prime $w_6 \in T$, generated by $\lambda_6$ (totally positive and $1$ modulo $8$), such that $w_6(\gamma) \equiv 1 \bmod 2$, and moreover $\alpha$ and $\beta$ are invertible squares modulo $w_6$. Such a place exists as $E$ is $2$-generic. We declare $s := r' + 17$, so we need to construct $17$ more prime ideals to define the auxiliary twist $\kappa$.

We will do this by adding to our list $\mathfrak{p}_1, \ldots, \mathfrak{p}_{r'}$ nine prime elements $p_{s - 16}, \dots, p_{s - 11}$ and $p_{s - 4}, p_{s - 3}, p_{s - 2}$ and eight prime ideals $\mathfrak{p}_{s - 10}, \dots ,\mathfrak{p}_{s - 5}$ and $\mathfrak{p}_{s - 1}, \mathfrak{p}_s$. We start our construction by imposing the following demands on the principal primes $p_i$ for $i \in \{s - 16, \dots, s - 11\}$

\begin{enumerate}
\item[$(C1)$] the ideals $\mathfrak{p}_{s - 16} := (p_{s - 16}), \dots, \mathfrak{p}_{s - 11} := (p_{s - 11})$ are distinct prime ideals coprime to $S$ and $\mathfrak{p}_1, \dots, \mathfrak{p}_{s - 17}$;
\item[$(C2)$] for each $i \in \{1, \dots , 6\}$, the elements $\lambda_i p_{s - 17 + i}$ are squares locally at $T - \{w_i\} \cup \{\mathfrak{p}_1, \dots , \mathfrak{p}_{s - 17}\}$, and $p_{s - 17 + i}$ is a square locally at $w_i$. Furthermore, for each $s - 16 \leq i < j \leq s - 11$ we have
$$
\left(\frac{\lambda_{j - s + 17} p_j}{p_i}\right) = 
\begin{cases}
-1 &\text{if } j = i + 1 \text{ and } i \in \{s - 16, s - 14, s - 12\} \\
1 &\text{otherwise};
\end{cases}
$$
\item[$(C3)$] we pick local uniformizers $\pi_{s - 16}, \dots, \pi_{s - 11}$ such that for all $i \in \{s - 16, \dots, s - 11\}$
$$
\pi_i = p_i \epsilon_i,
$$
where $\epsilon_i$ is a quadratic non-residue.
\end{enumerate}

\noindent It is possible to find such prime elements by repeatedly applying Mitsui's prime ideal theorem. Before we proceed with the construction of the prime ideals $\mathfrak{p}_{s - 10}, \dots, \mathfrak{p}_{s - 5}$, we quantify in a precise sense how $\lambda_1 p_{s - 16}, \dots, \lambda_6 p_{s - 11}$ do not give rise to Selmer elements. Define
$$
V := \{(x_1, x_2) \in (K^{\ast}/K^{\ast 2})^2 : x_1, x_2 \in \langle \lambda_1 p_{s - 16}, \dots , \lambda_6p_{s - 11} \rangle\}. 
$$
We claim that 
\begin{equation}
\label{eNoSelinp}
\mathrm{Sel}_{\mathcal{L}_{s - 11, \boldsymbol{\pi}}}(G_K , E[2]) \cap V = \{(1,1)\}.
\end{equation}
Note that every $(x_1, x_2) \in V$ has the shape
$$
(x_1, x_2) = \left((\lambda_1 p_{s - 16})^{\delta_1} \cdots (\lambda_6 p_{s - 11})^{\delta_6}, (\lambda_1 p_{s - 16})^{\delta_1'} \cdots (\lambda_6 p_{s - 11})^{\delta_6'}\right).
$$
Suppose now that $(x_1 , x_2) \in \mathrm{Sel}_{\mathcal{L}_{s - 11 , \boldsymbol{\pi}}}(G_K , E[2])$. Using the local conditions at $w_1, \dots, w_6$ (see also \eqref{eDeltaw1}, \eqref{eDeltaw2}, \eqref{eDeltaw3}, \eqref{eDeltaw4} and \eqref{eDeltaw5}), we deduce immediately the following constraints
$$
\delta_1 = \delta_1', \quad \delta_2' = 0, \quad \delta_3 = 0, \quad \delta_4' = 0, \quad \delta_5 = \delta_5', \quad \delta_6 = 0. 
$$
Hence we can rewrite our element $(x_1, x_2)$ as
\begin{align}
\label{ex1x2Shape}
x_1 &= (\lambda_1 p_{s - 16})^{\delta_1} (\lambda_2 p_{s - 15})^{\delta_2} (\lambda_4 p_{s - 13})^{\delta_4} (\lambda_5 p_{s - 12})^{\delta_5}, \nonumber \\
x_2 &= (\lambda_1 p_{s - 16})^{\delta_1} (\lambda_3 p_{s - 14})^{\delta_3'} (\lambda_5 p_{s - 12})^{\delta_5} (\lambda_6 p_{s - 11})^{\delta_6'}.
\end{align}
We make the subclaim that
\begin{align}
\label{eLocalSubclaim}
\{(x, x) \in \mathcal{L}_{s - 11, \boldsymbol{\pi}, p_{s - 16}}\} = \langle (\epsilon_{s - 16} p_{s - 16}, \epsilon_{s - 16} p_{s - 16}) \rangle.
\end{align}
To this end, first recall the choice of $\pi_{s - 16}$ in $(C3)$. Therefore it suffices to show that $-1$, $\beta$ and $\gamma$ are squares locally at $p_{s - 16}$. Since $p_{s - 16}$ is trivial locally at every $2$-adic and infinite place, we see that $-1$ is certainly a square locally at $p_{s - 16}$ by Hilbert reciprocity applied to the symbol $(-1, p_i)$. Using once more that $p_{s - 16}$ is trivial locally at every $2$-adic and infinite place, Hilbert reciprocity applied to respectively the symbols $(p_{s - 16}, \beta)$ and $(p_{s - 16}, \gamma)$ yields 
$$
\left(\frac{\beta}{p_{s - 16}}\right) = \prod_{\substack{v \in T \ \text{odd} \\ v(\beta) \equiv 1 \bmod 2}} (p_{s - 16}, \beta)_v = \prod_{\substack{v \in T \ \text{odd} \\ v(\beta) \equiv 1 \bmod 2}} (\lambda_1, \beta)_v,
$$
and
$$
\left(\frac{\gamma}{p_{s - 16}}\right) = \prod_{\substack{v \in T \ \text{odd} \\ v(\gamma) \equiv 1 \bmod 2}} (p_{s - 16}, \gamma)_v = \prod_{\substack{v \in T \ \text{odd} \\ v(\gamma) \equiv 1 \bmod 2}} (\lambda_1, \gamma)_v,
$$
where we have invoked the first part of $(C2)$ to rewrite the products. Recalling that $\lambda_1$ itself is totally positive and $1$ modulo $8$, another application of Hilbert reciprocity yields
$$
\prod_{\substack{v \in T \ \text{odd} \\ v(\beta) \equiv 1 \bmod 2}} (\lambda_1, \beta)_v = (\lambda_1, \beta)_{w_1} = 1 \quad \quad \text{ and } \quad \quad \prod_{\substack{v \in T \ \text{odd} \\ v(\gamma) \equiv 1 \bmod 2}} (\lambda_1, \gamma)_v = (\lambda_1, \gamma)_{w_1} = 1,
$$
where the last identity is by construction of the place $w_1$. This proves the subclaim \eqref{eLocalSubclaim}.

Having proven the subclaim \eqref{eLocalSubclaim}, we return to the proof of our claim \eqref{eNoSelinp}. We start by observing that
\begin{align}
\label{eSquaress16}
\lambda_3 p_{s - 14}, \quad \quad \lambda_4 p_{s - 13}, \quad \quad \lambda_5 p_{s - 12}, \quad \quad \lambda_6 p_{s - 11}
\end{align} 
are squares locally at $(p_{s - 16})$ thanks to $(C2)$. Moreover, $p_{s - 16}$ is a square locally at $\lambda_1$ by $(C2)$, and therefore $\lambda_1$ is a square locally at $p_{s - 16}$ by Lemma \ref{lFlip}. Inspecting equation \eqref{ex1x2Shape}, we conclude that $x_2$ becomes simply $p_{s - 16}^{\delta_1}$ locally at $(p_{s - 16})$. Hence the local conditions at $p_{s - 16}$, see equation \eqref{eLocalSubclaim}, forces that $\delta_1 = 0$. 

Arguing similarly for the local conditions imposed on $x_1$ at $p_{s - 15}$ and $p_{s - 13}$ (now using that $-1$, $\alpha$ and $\gamma$ are squares there) and for the local conditions imposed on $x_2$ at $p_{s - 14}$ and $p_{s - 12}$ (now using that $-1$, $\alpha$ and $\beta$ are squares there) gives that $\delta_2 = \delta_4 = \delta_5 = \delta_3' = 0$. Granted that $\delta_5 = 0$, we can now apply again the same argument for the local conditions imposed on $x_2$ at $p_{s - 11}$, and get that $\delta_6' = 0$. This verifies equation \eqref{eNoSelinp}.

\subsubsection*{Projecting the Selmer group}
To aid some of the forthcoming arguments, it will be helpful to introduce a map
$$
\text{proj}: \mathrm{Sel}_{\mathcal{L}_{s - 11, \boldsymbol{\pi}}}(G_K , E[2]) \to \mathrm{Sel}_{\mathcal{L}_{s - 17, \boldsymbol{\pi}}^{\text{ray}}}(G_K , E[2]),
$$
which is defined as follows: given $(a_1 , a_2) \in \mathrm{Sel}_{\mathcal{L}_{s - 11, \boldsymbol{\pi}}}(G_K , E[2])$, we uniquely decompose $a_1 = a_1(1)a_1(2)$ and $a_2 = a_2(1)a_2(2)$, where for each $i \in \{1, 2\}$ we have
$$
a_i(1) := \prod_{j = 1}^6 (\lambda_j p_{s - 17 +j})^{v_{(p_i)}(a_i)},
$$
and then we set $\text{proj}(a_1, a_2) := (a_1(2) , a_2(2))$. We observe that this is indeed an element of $\mathrm{Sel}_{\mathcal{L}_{s - 17, \boldsymbol{\pi}}^{\text{ray}}}(G_K , E[2])$, since each $\lambda_j p_{s - 17 +j}$ is a square locally at $\mathfrak{p}_1, \dots, \mathfrak{p}_{s - 17}$ by $(C2)$. It is evident that $\text{proj}$ is a linear map of $\mathbb{F}_2$-vectorspaces. Observe that this, combined with equation \eqref{eNoSelinp}, immediately implies that the operator $\text{proj}$ is \emph{injective}.

We will now explain how to construct the prime ideals $\mathfrak{p}_{s - 10}, \dots, \mathfrak{p}_{s - 5}$ along with local uniformizers $\pi_{s - 10}, \dots, \pi_{s - 5}$. We will construct said prime ideals and local uniformizers one by one, while always maintaining the two constraints
\begin{equation} 
\label{edonotleaveSel}
\mathrm{Sel}_{\mathcal{L}_{s - 11 + i + 1, \boldsymbol{\pi}}}(G_K , E[2]) \subseteq \mathrm{Sel}_{\mathcal{L}_{s - 11 + i,  \boldsymbol{\pi}}}(G_K , E[2])
\end{equation}
and 
\begin{equation} 
\label{eQuadConsistency}
\left(\frac{\lambda_{h_1} p_{s - 17 + h_1}}{\mathfrak{p}_{s - 11 + h_2}}\right) = 1 \quad \quad \text{ for each } (h_1 , h_2) \in \{1 , \dots, 6\}^2.
\end{equation}
Moreover, we shall ensure that the inclusion \eqref{edonotleaveSel} is strict if $\dim_{\FF_2} \mathrm{Sel}_{\mathcal{L}_{s - 11 + i,  \boldsymbol{\pi}}}(G_K , E[2]) \geq 3$. In order to explain the construction of $\mathfrak{p}_{s - 11 + i + 1}$ and $\pi_{s - 11 + i + 1}$, we distinguish two cases:

\begin{enumerate}
\item[\textbf{Case 1:}] Suppose first that $\mathrm{Sel}_{\mathcal{L}_{s - 11 +i , \boldsymbol{\pi}}}(G_K , E[2])$ is $1$-dimensional. Let $(x_1, x_2)$ be the unique non-trivial element of this space. In this case, we choose any prime ideal $\mathfrak{p}_{s - 11 + i + 1}$ such that the elements $\lambda_1 p_{s - 16}, \dots, \lambda_6 p_{s - 11}$ and $x_1, x_2$ are all squares locally at $\mathfrak{p}_{s - 11 + i + 1}$. Then we are either in the first or the third case of Lemma \ref{lSelmerChange}. Hence we can always choose $\pi_{s - 11 + i + 1}$ in such a way that we are in the third case. This will yield equality in equation \eqref{edonotleaveSel} (observe that $(x_1 , x_2)$ satisfies the local conditions at $\mathfrak{p}_{s - 11 + i + 1}$ by construction), and moreover we have enforced equation \eqref{eQuadConsistency} by construction. 

\item[\textbf{Case 2:}] Suppose that $\mathrm{Sel}_{\mathcal{L}_{s - 11 + i, \boldsymbol{\pi}}}(G_K , E[2])$ is not $1$-dimensional, so it is at least $3$-dimensional. Pick any two linearly independent elements $v_1$ and $v_2$ of $\mathrm{Sel}_{\mathcal{L}_{s - 11 + i, \boldsymbol{\pi}}}(G_K , E[2])$. Since the operator $\text{proj}$ is injective, we conclude that $\text{proj}(v_1)$ and $\text{proj}(v_2)$ are also linearly independent. Hence, in view of equation \eqref{eRaygeneric} and \cite[Lemma 5.8]{KP}, we know that there exists a prime $\mathfrak{q}$ with the property that the two vectors 
$$
\text{proj}(v_1)(\Frob_{\mathfrak{q}}), \quad \quad \text{proj}(v_2)(\Frob_{\mathfrak{q}})
$$ 
are linearly independent vectors in $\mathbb{F}_2^2$, where we view $\text{proj}(v_1)$ and $\text{proj}(v_2)$ as elements of $H^1(G_K, \mathbb{F}_2^2)$ via Kummer theory. In order to choose $\mathfrak{p}_{s - 11 + i + 1}$, we need one final modification of the prime $\mathfrak{q}$, which we explain next. By comparing ramification loci, we see that the coordinates of $\text{proj}(v_1)$ and $\text{proj}(v_2)$ span a space which intersects the span of $\lambda_1 p_{s - 16}, \dots, \lambda_6 p_{s - 11}$ trivially. Therefore there exists a prime ideal $\mathfrak{p}_{s - 11 + i + 1}$ such that 
$$
\text{proj}(v_1)(\Frob_{\mathfrak{q}}) = \text{proj}(v_1)(\Frob_{\mathfrak{p}_{s - 11 + i + 1}}), \quad \text{proj}(v_2)(\Frob_{\mathfrak{q}}) = \text{proj}(v_2)(\Frob_{\mathfrak{p}_{s - 11 + i + 1}})
$$
and with the additional requirement that all of the elements $\lambda_1 p_{s - 16}, \dots, \lambda_6 p_{s - 11}$ are squares modulo $\mathfrak{p}_{s - 11 + i + 1}$. We also pick any local uniformizer $\pi_{s - 11 + i +1}$. Now, by construction, we are in the second case of Lemma \ref{lSelmerChange}, which in particular implies that equation \eqref{edonotleaveSel} holds with the inclusion being strict. 
\end{enumerate}

\noindent Having completed the construction of $\mathfrak{p}_{s - 10}, \dots, \mathfrak{p}_{s - 5}$ and $\pi_{s - 10}, \dots, \pi_{s - 5}$, it is now time to describe $\mathrm{Sel}_{\mathcal{L}_{s - 5}, \boldsymbol{\pi}}(G_K , E[2])$. We claim that this space is a $1$-dimensional subspace of $\mathrm{Sel}_{\mathcal{L}_{s - 11, \boldsymbol{\pi}}}(G_K , E[2])$. Recall that $\mathrm{Sel}_{\mathcal{L}_{s - 17 , \boldsymbol{\pi}}}(G_K , E[2])$ is $1$-dimensional by equation \eqref{edimension1}. It follows from Lemma \ref{lSelmerChange} that 
$$
\dim_{\FF_2} \mathrm{Sel}_{\mathcal{L}_{s - 11, \boldsymbol{\pi}}}(G_K , E[2]) \leq 13, \quad \quad \dim_{\FF_2} \mathrm{Sel}_{\mathcal{L}_{s - 11, \boldsymbol{\pi}}}(G_K , E[2]) \equiv 1 \bmod 2.
$$
Therefore the above procedure clearly gives
$$
\dim_{\FF_2} \mathrm{Sel}_{\mathcal{L}_{s - 5, \boldsymbol{\pi}}}(G_K , E[2]) = 1, \quad \quad \mathrm{Sel}_{\mathcal{L}_{s - 5, \boldsymbol{\pi}}}(G_K , E[2]) \subseteq \mathrm{Sel}_{\mathcal{L}_{s - 11, \boldsymbol{\pi}}}(G_K , E[2]).
$$
Denote by $(a, b)$ the unique non-zero element of $\mathrm{Sel}_{\mathcal{L}_{s - 5, \boldsymbol{\pi}}}(G_K , E[2])$. Because of the containment $\mathrm{Sel}_{\mathcal{L}_{s - 5, \boldsymbol{\pi}}}(G_K , E[2]) \subseteq \mathrm{Sel}_{\mathcal{L}_{s - 11, \boldsymbol{\pi}}}(G_K , E[2])$ and the injectivity of $\text{proj}$, we know that $\text{proj}(a , b) = (v, w)$ is a non-trivial element of $\mathrm{Sel}_{\mathcal{L}_{s - 17 , \boldsymbol{\pi}}^{\text{ray}}}(G_K , E[2])$. Hence thanks to equation \eqref{eBasisSwap}, we know that $v \neq 1$. All in all, we have obtained that
$$
\mathrm{Sel}_{\mathcal{L}_{s - 5 , \boldsymbol{\pi}}}(G_K , E[2]) = \langle (a, b) \rangle, 
$$
where $(v, w) = \text{proj}(a, b)$ satisfies
\begin{equation}
\label{eNontrivialfirstcoordinate}
v \neq 1.    
\end{equation}

\subsubsection*{Constructing prime elements, part 2}
We are now ready to list our next three principal prime ideals. 
\begin{enumerate}
\item[$(C4)$] the ideals $ \mathfrak{p}_{s - 4} := (p_{s - 4}), \mathfrak{p}_{s - 3} := (p_{s - 3}), \mathfrak{p}_{s - 2} := (p_{s - 2})$ are distinct prime ideals coprime to $S$ and $\mathfrak{p}_1, \dots, \mathfrak{p}_{s - 5}$;
\item[$(C5)$] the prime elements $p_{s - 4}, p_{s - 3}, p_{s - 2}$ are squares locally at each place of $T$ and at each place in $\{\mathfrak{p}_1, \dots, \mathfrak{p}_{s - 17}\} \cup \{\mathfrak{p}_{s - 10}, \dots, \mathfrak{p}_{s - 5}\}$;
\item[$(C6)$] for each $s - 16 \leq i \leq s - 11$ and each $s - 4 \leq j \leq s - 2$ we have that
$$
\left(\frac{p_j}{p_i}\right) = -1 \Longleftrightarrow (i, j) \in \left\{ \begin{array}{l} (s - 16, s - 4), (s - 15, s - 4), (s - 14, s - 3), \\ (s - 13, s - 3), (s - 12, s - 2), (s - 11, s - 2) \end{array} \right\};
$$
\item[$(C7)$] for each $s - 4 \leq i < j \leq s - 2$, we have that
$$
\left(\frac{p_j}{p_i}\right) = 1.
$$
\end{enumerate}

\noindent It is clearly possible to find such prime elements by repeatedly applying Mitsui's prime ideal theorem. We next pick
\begin{enumerate}
\item[$(C8)$] local uniformizers $\pi_{s - 4}, \pi_{s - 3}, \pi_{s - 2}$ such that for all $i \in \{s - 4, s - 3, s - 2\}$
$$
\pi_i = p_i \epsilon_i,
$$
where $\epsilon_i$ is a quadratic non-residue.
\end{enumerate}

\noindent We claim that with this choice of primes and uniformizers, we have 
\begin{align}
\label{eniChange}
n_{s - 5}= n_{s - 4} = n_{s - 3} = 2. 
\end{align}
Furthermore, we claim that a basis for $\mathrm{Sel}_{\mathcal{L}_{s - 2, \boldsymbol{\pi}}}(G_K, E[2])$ is
\begin{align}
\{&(a, b), (\lambda_1 p_{s - 16}p_{s - 4}, \lambda_1 p_{s - 16}p_{s - 4}), (\lambda_2 p_{s - 15}p_{s - 4}, 1), (1, \lambda_3 p_{s - 14}p_{s - 3}), \nonumber \\
&(\lambda_4 p_{s - 13}p_{s - 3}, 1), (\lambda_5 p_{s - 12}p_{s - 2}, \lambda_5 p_{s - 12}p_{s - 2}), (1, \lambda_6p_{s - 11}p_{s - 2})\}. \label{eAlmostBasis}
\end{align}
Observe that equation \eqref{eAlmostBasis} and equation \eqref{eniChange} follow if we are able to show that the last six classes above, denoted $c_1, \dots, c_6$, are all in $\mathrm{Sel}_{\mathcal{L}_{s - 2, \boldsymbol{\pi}}}(G_K, E[2])$. The next subsection will be entirely dedicated to this verification.

\subsubsection*{Verifying local conditions}
We will distinguish five cases depending on $v \in T \cup \{\mathfrak{p}_1, \dots, \mathfrak{p}_{s - 2}\}$.

\begin{enumerate}
\item[\textbf{Case 1:}] Suppose that $v \in T - \{w_1, \dots, w_6\} \cup \{\mathfrak{p}_1, \dots, \mathfrak{p}_{s - 17}\}$. Observe that by the first part of $(C2)$ the elements $\lambda_1 p_{s - 16}, \dots, \lambda_6 p_{s - 11}$ are squares locally at $T - \{w_1, \dots, w_6\} \cup \{\mathfrak{p}_1, \dots, \mathfrak{p}_{s - 17}\}$. The same holds for $p_{s - 4}, p_{s - 3}, p_{s - 2}$, in view of $(C5)$. This gives that each $c_i$ is locally trivial at these places. 

\item[\textbf{Case 2:}] Suppose that $v \in \{w_1, \dots, w_6\}$. Using the first part of $(C2)$ and $(C5)$, we see that $c_i$ is trivial at $w_j$ for $i \neq j$ (for the same reason as Case 1). In virtue of the first part of $(C2)$ and of $(C5)$, the class $c_i$ locally at $w_i$ becomes precisely
\begin{alignat*}{3}
&\res_{w_1}(c_1) = \res_{w_1}(\lambda_1, \lambda_1), \quad &&\res_{w_2}(c_2) = \res_{w_2}(\lambda_2, 1), \quad &&\res_{w_3}(c_3) = \res_{w_3}(1, \lambda_3), \\ 
&\res_{w_4}(c_4) = \res_{w_4}(\lambda_4, 1), \quad &&\res_{w_5}(c_5) = \res_{w_5}(\lambda_5, \lambda_5), \quad &&\res_{w_6}(c_6) = \res_{w_6}(1, \lambda_6).
\end{alignat*}
These are local Selmer elements in virtue of our choice of $w_1, \dots, w_6$.

\item[\textbf{Case 3:}] Suppose that $v \in \{(p_{s - 16}), \dots, (p_{s - 11})\}$. Let us first examine the classes $c_i$ locally at $p_{s - 16}$. Right below equation \eqref{eSquaress16}, we have shown that 
$$
\lambda_1, \quad \lambda_3 p_{s -14}, \quad \lambda_4 p_{s -13}, \quad \lambda_5 p_{s - 12}, \quad \lambda_6 p_{s - 11}
$$ 
are squares locally at $p_{s - 16}$. Furthermore, the same holds for $p_{s - 3}$ and $p_{s - 2}$ by $(C6)$. This already forces $c_3, \dots, c_6$ to be locally trivial at $p_{s - 16}$. 

Invoking $(C6)$, we see that $p_{s - 4}$ is a non-square locally at $p_{s - 16}$. Therefore the class $c_1$ locally at $p_{s - 16}$ equals 
$$
(\epsilon_{s - 16}p_{s - 16}, \epsilon_{s - 16}p_{s - 16}).
$$
We have seen in equation \eqref{eLocalSubclaim} that this class is in the local Selmer space at $p_{s - 16}$. By $(C2)$, we have that $\lambda_2p_{s - 15}$ is a non-square at $p_{s - 16}$. Since $p_{s - 4}$ is also a non-square in view of $(C6)$, we see that $\lambda_2p_{s - 15}p_{s - 4}$ is a square locally at $p_{s - 16}$. Therefore $c_2$ is locally trivial at $p_{s - 16}$.

This completes the proof for $v = (p_{s - 16})$. Similar arguments work for the places $v \in \{(p_{s - 15}), \dots, (p_{s - 11})\}$.

\item[\textbf{Case 4:}] Suppose that $v \in \{\mathfrak{p}_{s - 10}, \dots, \mathfrak{p}_{s - 5}\}$. By equation \eqref{eQuadConsistency}, $\lambda_1 p_{s - 16}, \dots, \lambda_6 p_{s - 11}$ are squares locally at $v$. The same holds for $p_{s - 4}, p_{s - 3}, p_{s - 2}$ in view of $(C5)$. Hence each $c_i$ is trivial at these places. 

\item[\textbf{Case 5:}] Suppose that $v \in \{(p_{s - 4}), (p_{s - 3}), (p_{s - 2})\}$. Let us start with $p_{s - 4}$. Since the prime elements $p_{s - 4}, p_{s - 3}, p_{s - 2}$ are trivial at all $2$-adic and infinite places, we deduce from Lemma \ref{lFlip} and $(C7)$ that these three primes are all squares modulo each other. Invoking $(C5)$ and $(C6)$, the fact that $p_{s - 4}$ is totally positive and locally trivial above $2$, and Lemma \ref{lFlip}, we see that $\lambda_3p_{s - 14}, \dots, \lambda_6p_{s - 11}$ are all squares locally at $p_{s - 4}$ too. Hence it remains to check that the classes $c_1$ and $c_2$ satisfy the local conditions at $p_{s - 4}$. Invoking once more $(C5)$ and $(C6)$ along with Lemma \ref{lFlip}, we see that locally at $p_{s - 4}$ the elements $\lambda_1p_{s - 16}$ and $\lambda_2p_{s - 15}$ are both equal to $\epsilon_{s - 4}$. Hence, in view of $(C8)$, it suffices to check that $-1$, $\alpha$, $\beta$ and $\gamma$ are squares locally at $p_{s - 4}$. This follows immediately from $(C5)$ and Hilbert reciprocity.

For $p_{s - 3}$, we argue similarly to deduce that $\lambda_1 p_{s - 16}$, $\lambda_2 p_{s - 15}$, $\lambda_5 p_{s - 12}$ and $\lambda_6 p_{s - 11}$ are all squares locally at $p_{s - 3}$. Hence it remains to check the local conditions for $c_3$ and $c_4$. Note that $\lambda_3p_{s - 14}$ and $\lambda_4p_{s - 13}$ are both non-squares locally at $p_{s - 3}$ in view of $(C5), (C6)$ and Lemma \ref{lFlip}. Hence, in view of $(C8)$, it suffices to verify that $-1$, $\alpha$, $\beta$ and $\gamma$ are squares locally at $p_{s - 3}$, which is again a consequence of Hilbert reciprocity and $(C5)$.

For $p_{s - 2}$, note that $\lambda_1 p_{s - 16}$, $\lambda_2 p_{s - 15}$, $\lambda_3 p_{s - 14}$ and $\lambda_4 p_{s - 13}$ are all squares locally at $p_{s - 2}$. Hence we are left with $c_5$ and $c_6$. Moreover, $\lambda_5 p_{s - 12}$ and $\lambda_6 p_{s - 11}$ are both non-squares locally at $p_{s - 2}$ in view of $(C5), (C6)$ and Lemma \ref{lFlip}. Hence, in view of $(C8)$, it suffices to show that $-1$, $\alpha$, $\beta$ and $\gamma$ are squares locally at $p_{s - 2}$. This follows, once more, from $(C5)$ and Hilbert reciprocity.
\end{enumerate}

\subsubsection*{The final twist}
It will be notationally convenient to denote by $\delta: \{1, \dots, 6\} \to \{2, 3, 4\}$ the function
$$
\delta(1) = \delta(2) = 4, \quad \quad \delta(3) = \delta(4) = 3, \quad \quad \delta(5) = \delta(6) = 2.
$$
Before constructing the prime $\mathfrak{p}_{s - 1}$, we make the auxiliary claim that the elements 
$$
a, \ \ \ \lambda_1 p_{s - 16}p_{s - 4}, \ \ \ \lambda_2 p_{s - 15}p_{s - 4}, \ \ \ \lambda_3 p_{s - 14}p_{s - 3}, \ \ \ \lambda_4 p_{s - 13}p_{s - 3}, \ \ \ \lambda_5 p_{s - 12}p_{s - 2}, \ \ \ \lambda_6 p_{s -11}p_{s - 2}
$$
are linearly independent. Indeed, by exploiting the ramification loci of the prime elements $p_{s - 16}, \dots, p_{s - 11}$, we conclude that the only possible linear dependency would be
$$
1 = \frac{a}{\prod_{j = 1}^6 (\lambda_jp_{s - 17 +j}p_{s - \delta(j)})^{v_{(p_{s - 17 + j})}(a)}} = \frac{v}{\prod_{j=1}^{6}p_{s - \delta(j)}^{v_{(p_{s - 17 + j})}(a)}},
$$
which can be rewritten as
\begin{equation}
\label{evPrimeRel}
\prod_{j = 1}^6 p_{s - \delta(j)}^{v_{(p_{s - 17 + j})}(a)} = v.
\end{equation}
By definition of $\text{proj}$, we know that $v$ is ramified only at the primes in $T \cup \{\mathfrak{p}_1 , \dots , \mathfrak{p}_{s - 17}\}$. Therefore the relation \eqref{evPrimeRel} forces $v = 1$ contradicting \eqref{eNontrivialfirstcoordinate}, ending the proof of our claim.

With the claim established, we now select a prime $\mathfrak{p}_{s - 1}$ disjoint from the places in $S$ and $\{\mathfrak{p}_1, \dots, \mathfrak{p}_{s - 2}\}$ such that both $a$ and $\lambda_6 p_{s - 11}p_{s - 2}$ are non-squares, while 
$$
\lambda_1 p_{s - 16}p_{s - 4}, \quad \lambda_2 p_{s - 15}p_{s - 4}, \quad \lambda_3 p_{s - 14}p_{s - 3}, \quad \lambda_4 p_{s - 13}p_{s - 3}, \quad \lambda_5 p_{s - 12}p_{s - 2}
$$
are squares modulo $\mathfrak{p}_{s - 1}$. We also choose any local uniformizer $\pi_{s - 1}$. We are now visibly in the second case of Lemma \ref{lSelmerChange}. Using  equation \eqref{eAlmostBasis}, we conclude that a basis $\mathcal{B}$ for $\mathrm{Sel}_{\mathcal{L}_{s - 1 , \boldsymbol{\pi}}}(G_K, E[2])$ is
\begin{align}
\mathcal{B} := \{&(\lambda_1 p_{s - 16}p_{s - 4}, \lambda_1 p_{s - 16}p_{s - 4}), (\lambda_2 p_{s - 15}p_{s - 4}, 1), (1, \lambda_3 p_{s - 14}p_{s - 3}), \nonumber \\
&(\lambda_4 p_{s - 13}p_{s - 3}, 1) , (\lambda_5 p_{s - 12}p_{s - 2}, \lambda_5 p_{s - 12}p_{s - 2})\}. \label{eBasisB}
\end{align}
Let $\mathfrak{m}$ be the smallest modulus divisible by $8$ and by all places in $S$. Finally, we choose a prime $\mathfrak{p}_s \not \in S \cup \{\mathfrak{p}_1, \dots, \mathfrak{p}_{s - 1}\}$ such that the class of $\mathfrak{p}_1 \cdots \mathfrak{p}_s$ is trivial in $\mathrm{Cl}(K, \mathfrak{m})$ and such that 
$$
\lambda_1 p_{s - 16}p_{s - 4}, \quad \lambda_2 p_{s - 15}p_{s - 4}, \quad \lambda_3 p_{s - 14}p_{s - 3}, \quad \lambda_4 p_{s - 13}p_{s - 3}, \quad \lambda_5 p_{s - 12}p_{s - 2}
$$
are squares modulo $\mathfrak{p}_s$. We denote by $\kappa$ a generator of $\mathfrak{p}_1 \cdots \mathfrak{p}_s$ that is $1$ modulo $8$, is congruent to $1$ modulo every odd prime in $S$ and is totally positive. We finally choose $\pi_s$ in such a way to enforce that $n_{s - 1} = 0$. Therefore $\mathrm{Sel}_{\mathcal{L}_{s, \boldsymbol{\pi}}}(G_K, E[2])$ still has the basis $\mathcal{B}$.

\subsubsection*{Verification of $(K1)$ and $(K2)$} 
We now verify that $\kappa$ is an auxiliary twist, so we have to check that $\kappa$ satisfies $(K1)$ and $(K2)$ from Definition \ref{def: auxiliary triple}. 

We start by showing that $(K1)$ holds. By construction, $\kappa$ is congruent to $1$ modulo $8$, is totally positive and is congruent to $1$ modulo every odd prime in $S$. It follows that the character $\psi_\kappa$ is locally trivial at $S$. Hence $(K1)$ holds. 

We next deal with $(K2)$. Observe that, by construction, the ramification locus of $\psi_\kappa$ is $\mathfrak{p}_1, \dots, \mathfrak{p}_s$. Furthermore, with the choice of uniformizers $\pi_1, \dots, \pi_s$ made above, we have shown that $\mathrm{Sel}_{\mathcal{L}_{s, \boldsymbol{\pi}}}(G_K, E[2])$ has $\mathcal{B}$ as a basis (see \eqref{eBasisB}). We define 
\begin{alignat*}{2}
&(z_1, z_2) &&:= (\lambda_1 p_{s - 16}p_{s - 4}, \lambda_1 p_{s - 16}p_{s - 4}) \\
&(z_3, z_4) &&:= (\lambda_2 p_{s - 15}p_{s - 4}, 1) \\
&(z_5, z_6) &&:= (1, \lambda_3 p_{s - 14}p_{s - 3}) \\
&(z_7, z_8) &&:= (\lambda_4 p_{s - 13}p_{s - 3}, 1) \\
&(z_9, z_{10}) &&:= (\lambda_5 p_{s - 12}p_{s - 2}, \lambda_5 p_{s - 12}p_{s - 2}). 
\end{alignat*}
With this choice we readily see that equation \eqref{eInfMat1} holds, and this yields $(K2)$.
\end{proof}
\section{Proof of Corollary \ref{cRank1}} 
\label{sct: proof of cor}
In this section we will prove Corollary \ref{cRank1}. 

\begin{proof}[Proof of Corollary \ref{cRank1}]
By Theorem \ref{tMain}, it suffices to show that there exists $(a_1, a_2, a_3) \in O_K^3$ such that the elliptic curve 
$$
y^2 = (x - a_1) (x - a_2) (x - a_3)
$$
is $3$-generic. We take nine totally positive prime elements $p_1, \dots, p_9$ such that $|O_K/p_i| > 5$, $p_i \equiv 1 \bmod 8$ and
\begin{equation}
\label{eLeg1}
\left(\frac{p_j}{p_i}\right) = 1
\end{equation}
for each $1 \leq i < j \leq 9$. Clearly, such primes can be constructed by repeated application of Mitsui's prime ideal theorem. We deduce that
\begin{equation}
\label{eLeg2}
\left(\frac{p_i}{p_j}\right) = 1, \quad \quad \left(\frac{-1}{p_i}\right) = 1
\end{equation}
for each $1 \leq i < j \leq 9$, where the first equality follows from Lemma \ref{lFlip}, while the second follows from Hilbert reciprocity applied to $\psi_{p_i} \cup \psi_{-1}$ (note that the infinite and $2$-adic contributions vanish because the $p_i$ are totally positive and satisfy $p_i \equiv 1 \bmod 8O_K$). Set 
$$
a := p_1p_2p_3, \quad b := p_4p_5p_6, \quad c := p_7p_8p_9.  
$$
Observe that the conic
$$
aX^2 + bY^2 = cZ^2
$$
is, by construction, locally solvable at all places of $K$. Then, by scaling a solution with our nine principal primes if necessary, we see that it has a non-zero point $(X, Y, Z) \in O_K^3$ such that $\gcd(YZ, a) = (1)$, $\gcd(XZ, b) = (1)$ and $\gcd(XY, c) = (1)$. Now we put 
$$
a_1 := 0, \quad \quad a_2 := -aX^2, \quad \quad a_3 := -cZ^2. 
$$
We claim that the elliptic curve
$$
y^2 = (x - a_1) (x - a_2) (x - a_3)
$$
is $3$-generic. Observe that $\alpha = a X^2$, $\beta = c Z^2$ and $\gamma = b Y^2$. The primes $p_1$, $p_2$ and $p_3$ divide $\alpha$ an odd number of times. Moreover, $\pm \beta$ and $\pm \gamma$ are, by equations \eqref{eLeg1} and \eqref{eLeg2}, squares modulo $p_1$, $p_2$ and $p_3$. Arguing similarly for $p_4p_5p_6 \mid \beta$ and $p_7p_8p_9 \mid \gamma$ gives that $E$ is $3$-generic. 
\end{proof}

\end{document}